\documentclass[10pt,twoside, a4paper, english, reqno]{amsart}
\usepackage{graphicx, amsmath, varioref, amscd, amssymb,color, bm, stmaryrd}
\usepackage[pdftex, colorlinks=true, backref]{hyperref}
\usepackage[nobysame]{amsrefs}
\numberwithin{equation}{section}
\allowdisplaybreaks

\newtheorem{theorem}{Theorem}[section]
\newtheorem{lemma}[theorem]{Lemma}

\theoremstyle{definition}
\newtheorem{definition}[theorem]{Definition}

\theoremstyle{remark}
\newtheorem{remark}[theorem]{Remark}

\newcommand{\Pth}[1]{\partial_t^h\left(#1\right)}
\newcommand{\Div}{\operatorname{div}_x}
\newcommand{\Hdiv}{\vc{W}_0^{\text{div},2}(\Om)}
\newcommand{\Hcurl}{\vc{W}_0^{\text{curl},2}(\Om)}
\newcommand{\Curl}{\operatorname{curl}_x}
\newcommand{\Grad}{\nabla_x}
\newcommand{\vr}{\varrho}
\newcommand{\vu}{\vc{u}}
\newcommand{\vv}{\vc{v}}
\newcommand{\vw}{\vc{w}}

\newcommand{\vc}[1]{{\bm{#1}}}

\newcommand{\weak}{\rightharpoonup}

\newcommand{\norm}[1]{\left\Vert#1\right\Vert}
\newcommand{\abs}[1]{\left|#1\right|}
\newcommand{\Set}[1]{\left\{#1\right\}}
\newcommand{\jump}[1]{\left\llbracket #1\right\rrbracket}

\newcommand{\vrho}{\varrho}
\newcommand{\inb}{\in_{\text{b}}}
\newcommand{\Dt}{\Delta t}
\newcommand{\R}{\mathbb{R}}
\newcommand{\weakto}{\rightharpoonup}
\newcommand{\mcw}{\mathcal{W}}

\newcommand{\Om}{\ensuremath{\Omega}}
\newcommand{\cOm}{\ensuremath{\overline{\Omega}}}
\newcommand{\pOm}{\ensuremath{\partial\Omega}}
\newcommand{\Dom}{(0,T)\times\Omega}
\newcommand{\cDom}{[0,T)\times\overline{\Omega}}

\newcommand{\eff}{P_{\text{eff}}}
\newcommand{\binner}{\partial E \setminus \partial \Om}

\newcommand{\Aoph}[1]{\Pi_h^V \left(\Grad \Delta^{-1}\left[#1\right]\right)}
\newcommand{\Aop}[1]{\Grad \Delta^{-1}\left[#1\right]}

\newcommand{\avg}[1]{\langle #1 \rangle_\Om}

\newcommand{\solutiontext}
{
	Let $\Set{(\vrho_{h},\vc{w}_{h},\vc{u}_{h})}_{h>0}$ 
	be a sequence of numerical solutions 
	constructed according to \eqref{eq:num-scheme-II} 
	and Definition \ref {def:num-scheme}.
}

\begin{document}

\title[A convergent mixed method for the Stokes approximation equations]
{A convergent mixed method for the Stokes approximation of  viscous compressible flow}

\author[Kenneth H. Karlsen]{Kenneth H. Karlsen}
\address[Kenneth H. Karlsen]
{\newline 
Centre of Mathematics for Applications \newline 
University of Oslo \newline
P.O. Box 1053, Blindern \newline 
N--0316 Oslo, Norway\newline
and\newline
Department of Scientific Computing\newline
Simula Research Laboratory\newline
P.O.Box 134\newline
N--1325 Lysaker, Norway}

\email[]{kennethk@math.uio.no}
\urladdr{http://folk.uio.no/kennethk/}

\author{Trygve K. Karper} 
\address[Trygve K. Karper]
{\newline 
Centre of Mathematics for Applications \newline 
University of Oslo\newline 
P.O. Box 1053, Blindern\newline 
N--0316 Oslo, Norway}
\email[]{t.k.karper@cma.uio.no}
\urladdr{http://folk.uio.no/trygvekk/}

\thanks{This work was supported by the Research Council of Norway through
an Outstanding Young Investigators Award (K. H. Karlsen). 
This article was written as part of the the international research program
on Nonlinear Partial Differential Equations at the Centre for
Advanced Study at the Norwegian Academy of Science
and Letters in Oslo during the academic year 2008--09.}

\date{\today}

% \subjclass[2000]{Primary: 35D05, 65M12; Secondary: 65M06...... RETT OPP DETTE ................}

%\subjclass{............................35G25, 35L05, 65M06, 65M12...............................}

%\subjclass{Primary: 35G25, 35L05; Secondary: 65M06, 65M12}

% 35G25 Initial value problems for nonlinear higher-order PDE, nonlinear evolution equations

% 35L05 Wave equation

% 65M06 Finite difference methods 

% 65M12 Stability and convergence of numerical method

% \thanks{.......The research of K. H. Karlsen is supported by an Outstanding 
% Young Investigators Award (OYIA) from the
% Research Council of Norway. {\bf oppdater denne ....}}
% 
% \keywords{..........., convergence of approximate solutions {\bf oppdater denne med masse rart...}}

% NEW stuff

% 35Q30 Stokes and Navier-Stokes equations
% 74S05 Finite element methods
% 65M12 Stability and convergence of numerical methods


\maketitle
\begin{abstract}
We propose a mixed finite element method for the motion of a strongly viscous, ideal, and isentropic 
gas. At the boundary we impose a Navier--slip condition 
such that the velocity equation can be posed in mixed form with the 
vorticity as an auxiliary variable. 
In this formulation we design a finite element method, where the velocity and vorticity is approximated
with the div- and curl- conforming N\'ed\'elec elements, respectively, 
of the first order and first kind.
The mixed scheme is coupled to a standard piecewise constant upwind discontinuous Galerkin discretization
of the continuity equation.
For the time discretization, implicit Euler time stepping is used. 
Our main result is that the numerical solution converges to a weak solution 
as the discretization parameters go to zero. 
The convergence analysis is inspired by the continuous analysis 
of Feireisl and Lions for the compressible Navier--Stokes equations.
Tools used in the analysis include an equation for the effective viscous 
flux and various renormalizations of the density scheme.

\end{abstract}

\tableofcontents{}

\section{Introduction}
Let $\Om \subset \mathbb{R}^N$, $N=2$, $3$, be an open, convex, polygonal domain
with Lipschitz boundary $\pOm$ and let $T>0$ be a final time. 
We consider the flow of an ideal isentropic viscous gas governed 
by the \emph{Stokes approximation equations}
\begin{align}
	\partial_t\vr + \Div (\vr \vu) &=0, \quad \text{in }\Dom, \label{eq:contequation}\\
	\partial_t\vu  -  \mu \Delta \vu - \lambda \Grad \Div \vu  + \Grad p(\vr) &= 0, \quad \text{in }\Dom. \label{eq:momentumeq}
\end{align}
Here, the unknowns are the density $\vr = \vr(t,x) >0$ and  velocity $\vu = \vu(t,x) \in \mathbb{R}^N$.
The operators $\Grad$ and $\Div$ are respectively the spatial gradient and divergence operators, and
$\Delta = \Div \Grad$ is the Laplace operator. The viscosity coefficients $\mu$, $\lambda$ 
are assumed to be constant and to satisfy $\mu >0$, $N\lambda + 2\mu \geq 0$.

The pressure is given by \emph{Boyle's law} which in the isentropic
regime takes the form %\cite{Feireisl:2004oe}  %Need real reference
$p(\vr)= a\vr^\gamma$, where $a>0$ is constant. In real applications 
the value of $\gamma$ ranges from a maximum of $\frac{5}{3}$ for monoatomic gases, to values 
close to one for polyatomic gases at high temperatures. 
In this paper, we will for purely technical reasons
be forced to require $\gamma > \frac{N}{2}$. 
	
From the point of view of applications, the model \eqref{eq:contequation}--\eqref{eq:momentumeq} can be justified 
for flows at very low Reynolds numbers so that the effects of convection may be neglected.
It is also on the same form as various shallow water models \cite{Lions:1998ga}. 
From a mathematical perspective, the system \eqref{eq:contequation}--\eqref{eq:momentumeq} is 
a model problem containing some, but not all, of the difficulties associated with 
 compressible fluid dynamics.

Mathematical analysis concerning the well-posedness of the system \eqref{eq:contequation}--\eqref{eq:momentumeq}
seems to originate with the papers \cite{KZEN,KZTO} by Kazhikov and collaborators. 
Several other contributions on the existence and long term stability 
exist, also in the context of similar shallow water models.
However, for our purpose here, the most relevant 
study is that of Lions \cite{Lions:1998ga} in which the global existence of (weak) solutions
and some higher regularity results are established.

In this paper we impose the following boundary conditions:
\begin{align}
		\vu \cdot \nu &= 0, \quad \text{on }(0,T)\times \pOm, \label{eq:bc-normal}\\
		\Curl \vu \times \nu & = 0, \quad \text{on }(0,T)\times \pOm, \label{eq:bc-navierslip}
\end{align}
where $\nu$ is the unit outward normal on $\pOm$ and $\Curl$ is the curl operator.
Here, in 2D, $\Curl$ denotes the rotation operator taking vectors into scalars.
The first condition is a natural condition of impermeability type on the normal velocity.
The second condition is in the literature commonly referred to as the Navier--slip condition.
While these boundary conditions are not motivated by physics, they are widely 
used in numerical methods. In particular, in the context of geophysical flows they 
are often preferred over classical Dirichlet conditions since the latter necessitates expensive calculations
of boundary layers. Of more importance to this paper, the boundary conditions \eqref{eq:bc-normal}--
\eqref{eq:bc-navierslip} allow us to pose the system \eqref{eq:contequation}--\eqref{eq:momentumeq}
in mixed form with $\Curl \vu$ as an auxiliary variable. This fact will play a crucial 
role in the upcoming analysis.

While many numerical methods appropriate for the Stokes approximation and 
Navier--Stokes equations have been proposed, the convergence properties 
of these methods are mostly unsettled. In fact, it is not clear 
whether or not any of these methods, in more than one dimension, 
converge to a (weak) solution as discretization parameters tend to zero.
In one dimension, there are some available results due to 
D. Hoff and his collaborators. However, these results apply 
to an ideal gas in Lagrangian coordinates and with initial 
data of bounded variation. In more than one dimension, 
there are some recent results for simplified models. 
In the papers \cite{Gallouet1, Gallouet2}, a convergent finite element method 
for a stationary compressible Stokes system is proposed and analyzed. 
The system considered there are similar to \eqref{eq:contequation}--\eqref{eq:momentumeq}
but without temporal dependence. 
In \cite{Karlsen2}, we established convergence of a finite 
element method for a semi--stationary version of \eqref{eq:contequation}--\eqref{eq:momentumeq} 
($\partial_t \vu = 0$) and homogenous Dirichlet boundary conditions.
This paper can be seen as a continuation of the recent study
\cite{Karlsen1} in which a convergent numerical method 
for the same semi--stationary system (\eqref{eq:contequation}--\eqref{eq:momentumeq} with $\partial_t \vu = 0$)
 with boundary conditions \eqref{eq:bc-normal}--\eqref{eq:bc-navierslip} 
was established. 
The main novelty of this paper is consequently the addition of 
the time derivative term $\partial_t \vu$ in the velocity equation 
\eqref{eq:momentumeq}.
%  At a first glance, 
% this   . This is however not the case as will become 
% evident in the upcoming analysis.

Let us now discuss our choice of numerical method for the Stokes 
approximation equations. 
For the time discretization, we will use implicit time stepping 
in both equations. 
To approximate the continuity equation \eqref{eq:contequation} 
we will use a standard piecewise constant upwind 
discontinuous Galerkin method. 
To approximate the velocity, we will use a mixed finite element 
method with the N{\'e}d{\'e}lec's spaces of the first order and first kind. 
The mixed formulation is motivated by introducing  
the vorticity $\vc{w} = \Curl \vu$ as an auxiliary 
unknown and recasting the velocity equation \eqref{eq:momentumeq} 
in the form:
\begin{equation}\label{eq:vorticity-form}
\partial_t \vu	+ \mu \Curl \vc{w} - (\lambda+\mu) \Grad \Div \vu + \Grad p(\vr) = 0,
\end{equation}
where the identity $-\Delta = \Curl \Curl - \Grad \Div$ is used. 
This leads to a natural mixed formulation in which 
the requirement $\vc{w} = \Curl \vu$ plays the role of a lagrangian 
multiplier. 

Denote by $\Hdiv$ the vector fields $\vc{u}$ on $\Omega$ 
for which $\Div \vc{u}\in L^2$ and $\vc{u} \cdot \nu|_{\partial \Omega} = 0$, and by 
$\Hcurl$ the vector fields $\vc{w}$ on $\Omega$ for which 
$\Curl \vc{w}\in L^2$ and $\vc{w} \times \nu|_{\partial \Omega}=0$. 
We choose corresponding finite element spaces $\vc{V}_{h} \subset \Hdiv$ 
and $\vc{W}_{h}\subset \Hcurl$ based on 
N{\'e}d{\'e}lec's elements of the first order and first kind \cite{Nedelec:1980ec}.
The mixed finite element methods seeks, for each time step $k=1, \ldots, M$, functions 
$(\vw_h^k,\vu_h^k) \in \vc{W}_h \times \vc{V}_h$
such that 
\begin{equation}\label{intro:scheme}
	\begin{split}
		&\int_\Om \Pth{\vu_h^k}\vv_h  + \mu \Curl \vw^k_h \vv_h 
+ \left[(\lambda + \mu)\Div \vu_h^k  - p(\vr_h^k)\Div \vv_h\right]~dx = 0, \\
		&\int_\Om \vw_h^k \vc{\eta}_h -  \vu_h^k\Curl \vc{\eta}_h ~dx = 0,
	\end{split}
\end{equation}
for all $(\vv_h, \vc{\eta}_h)\in \vc{W}_h \times \vc{V}_h$, where $\vr^k_h$ is 
given and $\Pth{\vu_h^k}= (\Delta t)^{-1}[\vu_h^k - \vu_h^{k-1}]$ denotes 
implicit time stepping. Note that the boundary conditions 
\eqref{eq:bc-normal}--\eqref{eq:bc-navierslip} are mandatory to 
obtain this formulation.

Our main result is that $\{(\vw_h, \vu_h, \vr_h)\}_{h>0}$ converges 
 to a weak solution of the \emph{Stokes approximation equations}, at least along a subsequence.
The major difficulty is to obtain strong compactness of 
the density approximation $\{\vr_h\}_{h>0}$ which is needed in order 
to pass to the limit in the nonlinear pressure function.  
Since the density approximations 
are only  bounded in $L^\infty(0,T;L^\gamma(\Om))$ this is intricate.
At the heart of the convergence analysis lies the \emph{effective viscous flux} 
$
	\eff(\vr_h, \vu_h)= p(\vr_h) - (\lambda + \mu)\Div \vu_h.
$
In particular, strong convergence of the density approximation
follows from the property:
\begin{equation}\label{intro:effvisc}
	\lim_{h \rightarrow 0}\int \int \psi \eff(\vr_h, \vu_h) \vr_h~dxdt = \int \int \psi\overline{\eff(\vr, \vu)} \vr~dxdt, 
\end{equation}
for all $\psi \in C_c^\infty(0,T)$. 
It is in the process of obtaining \eqref{intro:effvisc} that 
the carefully selected finite element spaces and mixed form 
prove useful. Specifically, we obtain \eqref{intro:effvisc} by
setting $\vc{v}_h = \Pi_h^V \Grad \Delta^{-1}\vr_h$ in \eqref{intro:scheme}, 
where $\Pi_h^V$ is the canonical interpolation operator into $\vc{V}_h$. 
This test function satisfies $\Div \vv_h = \vr_h$ and 
is almost orthogonal to curls. 
The main difficult in obtaining \eqref{intro:effvisc} is to
treat the time derivative term, which,
with $\vc{v}_h$ as described above, is of the form
\begin{equation*}
	\begin{split}
		\int\int \Pth{\vu_h} \Pi_h^V\Grad \Delta^{-1}[\vr_h]~dxdt
		&=  \int\int \Pth{\vu_h} \Grad \Delta^{-1}[\vr_h]~dxdt + O(h)	 \\
		&= \int \int \Delta^{-1}\left[\Div \vu_h\right]\Pth{\vr_h}~dxdt + O(h).
	\end{split}
\end{equation*}
Using the continuity scheme, the last term 
the last term can be shown to converge. The property \eqref{intro:effvisc}
then follows.
Our analysis  resembles that of Lions and Feireisl for 
the compressible Navier--Stokes equations. 

As part of the analysis, we will need that the discrete velocity 
$\vu_h$ converges strongly to a function $\vu$. 
This is not immediate since the  approximation space $\vc{V}_h$ is 
only $\Div$ conforming. To obtain strong convergence, we utilize the discrete Hodge
decomposition $\vc{V}_h = \Curl \vc{W}_h + \vc{V}^{0,\perp}_h$
satisfied by the chosen N{\'e}d{\'e}lec spaces. 
When writing $\vu_h = \Curl \vc{\zeta}_h + \vc{z}_h$, it can be seen 
that $\vc{\zeta}_h$ does not depend on the density $\vr_h$ and as 
a consequence converges strongly. The remaining term $\vc{z}_h$ 
is then  weakly discrete curl free with bounded divergence and 
an estimate from the previous paper \cite{Karlsen1} yields
\begin{equation}\label{eq:jassa1}
\|\vc{z}_h(t, x)- \vc{z}_h(t, x - \xi)\|_{L^2(0,T;\vc{L}^2(\Om))} \rightarrow 0, ~\text{ as }|\xi| \rightarrow 0,
\end{equation}
uniformly in $h$. From the velocity scheme, we deduce 
a weak time-continuity estimate of the form
\begin{equation}\label{eq:jassa2}
	\Pth{\vc{z}_h} \in L^1(0,T;W^{-1,1}(\Om)),
\end{equation}
independently of $h$.
The two estimates \eqref{eq:jassa1} and \eqref{eq:jassa2} tells us that $\vc{z}_h$
satisfies the hypotheses of an Aubin--Lions type lemma (see Lemma \ref{lemma:aubinlions} below for details).
Strong convergence of $\vc{z}_h$ follows from this lemma. 

The paper is organized as follows: In Section \ref{sec:prelim}, we introduce notation and 
list some basic results needed for the later analysis. 
Moreover, we recall the usual notion of weak solution 
and introduce a mixed weak formulation of the velocity equation. 
Finally, we introduce the finite element spaces and review some of their basic properties. 
In Section \ref{sec:numerical-method}, we present the numerical method and state our main 
convergence result. Section \ref{sec:basic-est} is 
devoted to deriving basic estimates.
In Section \ref{sec:higherint}, we establish higher integrability 
of the density.  
Finally, in Section \ref{sec:conv}, we prove the main convergence 
result stated in Section \ref{sec:numerical-method}. 
The proof is divided into several steps (subsections), including 
convergence of the continuity scheme, weak continuity 
of the discrete viscous flux, strong convergence 
of the density approximations, and convergence of the velocity scheme.

\section{Preliminary material}\label{sec:prelim}
We will write $W^{m,p}(\Om)$ for the Sobolev space of functions
with derivatives of all orders up to $m$ belonging to the space
$L^p(\Om)$. To distinguish between scalar and vector functions, we will write 
vector functions with a bold face. Similarly, a functions space written 
in bold face denotes the vector analog of the corresponding scalar space. 

We make frequent use of the divergence and curl operators and 
denote these by $\Div$ and $\Curl$, respectively.
In the 2D case, we will denote both the rotation operator taking 
scalars into vectors 
and the curl operator taking vectors into scalars by $\Curl$.

We will  make use of the spaces
\begin{align*}
	L^2_0(\Om) &= \Set{\phi \in L^2(\Om): \int_\Om \phi~dx = 0}, \\
	\vc{W}^{\text{div}, 2}(\Om) & = 
	\Set{\vc{v} \in \vc{L}^2(\Omega): \Div \vc{v} \in L^2(\Omega)}, \\
	\vc{W}^{\text{curl},2}(\Om)& = 
	\Set{\vc{v} \in \vc{L}^2(\Omega): \Curl \vc{v} \in \vc{L}^2(\Omega)},
\end{align*}
where $\nu$ denotes the  outward pointing unit normal vector on $\partial \Omega$. 
If $\vc{v}\in \vc{W}^{\text{div}, 2}(\Omega)$ satisfies $\vc{v} \cdot \nu|_{\partial \Omega}=0$, we 
write $\vc{v}\in \Hdiv$. Similarly, $\vc{v}\in \Hcurl$ 
means $\vc{v}\in\vc{W}^{\text{curl}, 2}(\Omega)$ and $\vc{v} \times \nu|_{\partial \Omega} = 0$. 
In two dimensions, $\vc{w}$ is a scalar function and the 
space $\Hcurl$ is to be 
understood as $W_{0}^{1,2}(\Omega)$. To define 
weak solutions, we shall use the space
$$
\mcw(\Om) =\Set{\vc{v}\in  \vc{L}^2(\Omega): 
\Div \vc{v} \in L^2(\Omega), 
\Curl \vc{v} \in \vc{L}^2(\Omega), 
\vc{v} \cdot \nu|_{\partial \Omega}=0},
$$
which coincides with $ \Hdiv \cap \Hcurl$. 
The space $\mcw(\Om)$ is equipped with the norm $\norm{\vc{v}}_{\mcw}^2=\norm{\vc{v}}_{\vc{L}^2(\Omega)}^2
+\norm{\Div \vc{v}}_{\vc{L}^2(\Omega)}^2+\norm{\Curl \vc{v}}_{\vc{L}^2(\Omega)}^2$. 
It is known that $\norm{\cdot}_{\mcw}$ is equivalent to the $\vc{W}^{1,2}$ norm 
on the space $\Set{v\in \vc{W}^{1,2}(\Omega): \vc{v} 
\cdot \nu|_{\partial \Omega}=0}$, see, e.g., \cite{Lions:1998ga}.

% The space $\mcw(\Omega)$ admits a unique orthogonal Hodge decomposition 
% \begin{equation}\label{eq:hodge-cont}
% 	\mcw(\Omega) = \Curl S(\Omega) + D\Delta^{-1}L^2_{0}(\Omega),
% \end{equation}
% where $S(\Omega)= \{\vc{v} \in \vc{W}^{1,2}(\Omega): 
% \Curl \vc{v} \in \vc{W}^{1,2}(\Omega)\}$, $\Delta^{-1}$ 
% is the inverse Neumann Laplace operator.

For the convenience of the reader we list some basic 
functional analysis results to be utilized (often without mentioning) 
in the subsequent arguments (for proofs, see, e.g.,
\cite{Feireisl:2004oe}). Throughout the paper we 
use overbars to denote weak limits.

\begin{lemma}\label{lem:prelim} 
Let $O$ be a bounded open subset of $\R^M$ with $M\ge 1$.  
Suppose $g\colon \R\to (-\infty,\infty]$ is a lower semicontinuous 
convex function and $\Set{v_n}_{n\ge 1}$ is a sequence of 
functions on $O$ for which $v_n\weakto v$ in $L^1(O)$, $g(v_n)\in L^1(O)$ for each 
$n$, $g(v_n)\weakto \overline{g(v)}$ in $L^1(O)$. Then 
$g(v)\le \overline{g(v)}$ a.e.~on $O$, $g(v)\in L^1(O)$, and
$\int_O g(v)\ dy \le \liminf_{n\to\infty} \int_O g(v_n) \ dy$. 
If, in addition, $g$ is strictly convex on an open interval
$(a,b)\subset \R$ and $g(v)=\overline{g(v)}$ a.e.~on $O$, 
then, passing to a subsequence if necessary, 
$v_n(y)\to v(y)$ for a.e.~$y\in \Set{y\in O\mid v(y)\in (a,b)}$.
\end{lemma}

Let $X$ be a Banach space and denote by $X^\star$ its dual.  The space
$X^\star$ equipped with the weak-$\star$ topology is denoted by
$X^\star_{\mathrm{weak}}$, while $X$ equipped with the weak topology
is denoted by $X_{\mathrm{weak}}$. By the Banach-Alaoglu theorem, a
bounded ball in $X^\star$ is $\sigma(X^\star,X)$-compact.  If $X$
separable, then the weak-$\star$ topology is metrizable on bounded
sets in $X^\star$, and thus one can consider the metric space
$C\left([0,T];X^\star_{\mathrm{weak}}\right)$ of functions $v:[0,T]\to
X^\star$ that are continuous with respect to the weak topology. We
have $v_n\to v$ in $C\left([0,T];X^\star_{\mathrm{weak}}\right)$ if
$\langle v_n(t),\phi \rangle_{X^\star,X}\to \langle v(t),\phi
\rangle_{X^\star,X}$ uniformly with respect to $t$, for any $\phi\in
X$. The following lemma is a consequence of the Arzel\`a-Ascoli
theorem:

\begin{lemma}\label{lem:timecompactness}
Let $X$ be a separable Banach space, and suppose $v_n\colon [0,T]\to
X^\star$, $n=1,2,\dots$, is a sequence for which 
$\norm{v_n}_{L^\infty([0,T];X^\star)}\le C$, for some constant $C$ independent of $n$. 
Suppose the sequence $[0,T]\ni t\mapsto \langle v_n(t),\Phi \rangle_{X^\star,X}$, $n=1,2,\dots$, 
is equi-continuous for every $\Phi$ that belongs to a dense subset of $X$.  
Then $v_n$ belongs to $C\left([0,T];X^\star_{\mathrm{weak}}\right)$ for every
$n$, and there exists a function $v\in 
C\left([0,T];X^\star_{\mathrm{weak}}\right)$ such that along a 
subsequence as $n\to \infty$ there holds $v_n\to v$ in 
$C\left([0,T];X^\star_{\mathrm{weak}}\right)$.
\end{lemma}

In what follows, we will often obtain a priori estimates for a sequence $\Set{v_n}_{n\ge 1}$ 
that we write as ``$v_n\inb X$'' for some functional space $X$. What this really means is that 
we have a bound on $\norm{v_n}_X$ that is independent of $n$.

The following discrete version of a lemma due to Lions \cite[Lemma 5.1]{Lions:1998ga} 
will prove useful in the convergence analysis. A proof of this lemma can be found in \cite{Karlsen2}.
\begin{lemma}\label{lemma:aubinlions}
Given $T>0$ and a small number $h>0$, write 
$(0,T] = \cup_{k=1}^M(t_{k-1}, t_{k}]$ with $t_{k} = hk$ and $Mh = T$. 
Let $\{f_{h}\}_{h>0}^\infty$, $\{g_{h}\}_{h>0}^\infty $ be two sequences such that:
\begin{enumerate}
\item{} 
the mappings $t \rightarrow g_{h}(t,x)$ and $t \rightarrow f_{h}(t,x)$ 
are constant on each interval $(t_{k-1}, t_{k}],\ k=1, \ldots, M$.

\item{}$\{f_{h}\}_{h>0}$ and $\{g_{h}\}_{h>0}$ converge weakly to 
$f$ and $g$ in $L^{p_{1}}(0,T;L^{q_{1}}(\Om))$ and 
$L^{p_{2}}(0,T;L^{q_{2}}(\Om))$, respectively, where $1 < p_{1},q_{1}< \infty$ and
$$
\frac{1}{p_{1}} + \frac{1}{p_{2}} = \frac{1}{q_{1}} + \frac{1}{q_{2}} = 1.
$$

\item{}
the discrete time derivative satisfies
$$
\frac{g_{h}(t,x) - g_{h}(t-h,x)}{h} \in_{b} L^1(0,T;W^{-1,1}(\Om))
$$

\item{}$\|f_{h}(t,x) - f_{h}(t,x-\xi)\|_{L^{p_{2}}(0,T;L^{q_{2}}(\Om))} 
\rightarrow 0$ as $|\xi|\rightarrow 0$, uniformly in $h$.
\end{enumerate}

Then $g_{h}f_{h} \weak gf$ in the sense of distributions on $\Dom$.
\end{lemma}

\subsection{Weak and renormalized solutions}

\begin{definition}[Weak solutions]\label{def:weak}

We say that a pair $(\varrho,\vc{u})$ of functions constitutes a weak solution
of the Stokes approximation equations \eqref{eq:contequation}--\eqref{eq:momentumeq} 
with initial data 
$$
	(\vr^0, \vu^0) \in L^\gamma(\Om)\times \vc{L}^2(\Om), \quad \gamma > \frac{N}{2},
$$ 
and Navier-slip type boundary conditions 
\eqref{eq:bc-normal}--\eqref{eq:bc-navierslip}, 
provided the following conditions hold:

\begin{enumerate}
	\item $(\varrho,\vc{u}) \in L^\infty(0,T;L^\gamma(\Omega))\times L^2(0,T;\mcw)\cap L^\infty(0,T;\vc{L}^2(\Om))$;

	\item $\varrho_{t} + \Div (\varrho\vc{u}) = 0$ in the weak 
	sense, i.e, $\forall \phi \in C^\infty([0,T)\times\overline{\Omega})$,
	\begin{equation}\label{eq:weak-rho}
		\int_{0}^T\int_{\Omega}\varrho \left( \phi_{t} + \vc{u}\Grad \phi\right)\ dxdt
		+ \int_{\Omega}\vrho^0\phi|_{t=0}\ dx = 0;
	\end{equation}
		
	\item $\vc{u}_{t}-\mu \Delta \vc{u} - \lambda \Grad \Div\vc{u} + \Grad p(\varrho) =0$ weakly,
	i.e, $\forall \vc{\phi} \in \vc{C}^\infty([0,T)\times\overline{\Omega})$ for which 
	$\vc{\phi} \cdot \nu = 0$ on $(0,T)\times \partial \Omega$,
	\begin{equation}\label{eq:weak-u}
		\begin{split}
			&\int_{0}^T\int_{\Omega}-\vc{u}\vc{\phi_{t}} + \mu\Curl \vc{u} \Curl \vc{\phi} \\
			&\qquad \qquad + \left[(\mu + \lambda)\Div \vc{u}-p (\varrho)\right]\Div \vc{\phi}\ dxdt 
			=\int_{\Omega}\vc{u}^{0}\vc{\phi}|_{t=0} \ dx.			
		\end{split}
	\end{equation}
\end{enumerate}
\end{definition}

For the convergence analysis we shall also need the DiPerna-Lions concept of 
renormalized solutions of the continuity equation.  

\begin{definition}[Renormalized solutions]
\label{renormlizeddef}
Given $\vc{u}\in L^2(0,T;\mcw(\Om))$, we say that $\vrho\in  L^\infty(0,T;L^\gamma(\Omega))$ 
is a renormalized solution of \eqref{eq:contequation} provided 
$$
B(\vrho)_t + \Div \left(B(\vrho)\vc{u}\right) + b(\vrho)\Div \vc{u}=0
\quad \text{in the weak sense on $\cDom$,}
$$
for any $B\in C[0,\infty)\cap C^1(0,\infty)$ with 
$B(0)=0$ and $b(\vrho) := \vrho B'(\vrho) - B(\vrho)$.
\end{definition}

We shall need the following lemma. A proof can be found in \cite{Karlsen1}.
\begin{lemma}\label{lemma:feireisl}
Suppose $(\vrho,\vc{u})$ is a weak solution according to Definition \ref{def:weak}.
If $\vrho \in L^2((0,T)\times \Omega))$, then $\vrho$ is a renormalized solution 
according to Definition \ref{renormlizeddef}.
\end{lemma}

\subsection{A mixed formulation}
In view of the Navier--slip boundary condition \eqref{eq:bc-navierslip}
the velocity equation \eqref{eq:momentumeq}
admits the following mixed weak formulation, which we will use
to design a mixed finite element method: Determine functions
$$
(\vc{w},\vc{u}) \in L^2(0,T;\Hcurl)
\times L^2(0,T;\Hdiv)
$$
such that
\begin{equation}\label{def:mixed-weak}
	\begin{split}
		& \int_{0}^T \int_{\Omega}- \vc{u}\vc{v}_{t} + \mu\Curl \vc{w} \vc{v}
		+ \left[(\mu + \lambda)\Div \vc{u} - p(\varrho)\right]\Div \vc{v}\ dxdt
		= \int_{\Omega}\vc{u}^{0}\vc{v}|_{t=0}\ dx, \\
		& \int_{0}^T\int_{\Omega} \vc{w}\vc{\eta}-\Curl \vc{\eta} \vc{u}\ dxdt = 0, 
	\end{split}
\end{equation}
for all $(\vc{\eta},\vc{v}) \in  L^2(0,T;\Hcurl)
\times L^2(0,T;\Hdiv) \cap W^{1,2}(0,T;\vc{L}^2(\Omega))$. 

Note that if $(\vw, \vu,\vr)$ is a triple satisfying  the mixed formulation \eqref{def:mixed-weak}
then the pair $(\vu, \vr)$ satisfies the weak formulation \eqref{def:weak} \cite{Karlsen1}.

% In the upcoming convergence analysis we prove 
% that the approximate solutions, denoted by 
% $(\vrho_h,\vc{w}_h,\vc{u}_h)$, converge (in a sense made precise in forthcoming 
% sections) along a subsequence as $h\to$ to a limit $(\vrho,\vc{w},\vc{u})$ 
% satisfying \eqref{eq:weak-rho} and \eqref{def:mixed-weak}. Having constructed such
% a limit triple, it then follows immediately that the limit pair $(\vrho,\vc{u})$ is  a weak solution 
% according to Definition \ref{def:weak}, thereby completing the analysis.

\subsection{Finite Element spaces and basic results}
Throughout this paper, $\{E_{h}\}_{h}$ denotes a shape regular 
family of tetrahedral meshes of $\Omega$,  where $h$ is the maximal diameter.
By shape regular we mean that there exists a constant 
$\kappa > 0$ such that every $E \in E_{h}$ contains a ball of radius 
$\lambda_{E} \geq \frac{h_{E}}{\kappa}$, where $h_{E}$ is the diameter of $E$. 
For each fixed $h>0$, we let $\Gamma_h$ denote the set of faces in $E_h$
and $\mathcal{V}_h$ the set of edges. In two dimensions, 
$\Gamma_h$ is the set of edges and $\mathcal{V}_h$ the 
set of vertices. We will use $\mathbb{P}_j^k(E)$ to 
denote the space of vector polynomials on $E$ with $l$ components
and maximal order $k$.

To approximate the vorticity $\vc{w}$, we will 
use the curl-conforming N\'ed\'elec space of the first
order and kind \cite{Nedelec:1980ec}:
\begin{equation}\label{eq:Wh-def}
	\begin{split}
		\vc{W}_h(\Om)= &\left\{ \vc{w} \in \vc{W}^{\operatorname{curl},2}_0(\Om):
			~\vc{w}|_E \in \vc{W}(E), \,\forall E \in E_h,\right. \\
		&\qquad\qquad\qquad\quad
			\left.~ \int_e \jump{\vc{w}\cdot t}_e~dS(x) = 0,
			 \,\forall e \in \mathcal{V}_h\right\},		
	\end{split}
\end{equation}
where $t$ is the unit tangential along the edge $e$, $\jump{\cdot}_e$ 
is the jump over the edge $e$, and
\begin{equation*}
	\vc{W}(E) = 
	\begin{cases}
	\mathbb{P}_1^1(E), & N=2, \\	
	\Set{\vc{w} \in  \mathbb{P}^3_1(E): ~\Grad \vc{w} + \Grad\vc{w}^T = 0}, & N=3.
\end{cases}
\end{equation*}
In two dimensions, the continuity requirement $\int_e \jump{\vc{w}\cdot t}~dS(x) = 0$ in \eqref{eq:Wh-def}
is to be understood as continuity at  vertices.
For the velocity $\vu$, we will use the div-conforming
N\'ed\'elec space of the first
order and kind \cite{Nedelec:1980ec}:
\begin{equation*}
	\vc{V}_h(\Om) = \Set{\vc{v} \in \vc{W}^{\operatorname{div},2}_0:
					\, \vc{v}|_E \in \vc{V}(E), ~\forall E \in E_h,
					~\int_\Gamma \jump{\vc{v}\cdot \nu}~dS(x) = 0, ~\forall \Gamma \in \Gamma_h},
\end{equation*}
where
$
\vc{V}(E) = \mathbb{P}_0^N \oplus \mathbb{P}_0^1 \vc{x},
$
and $\jump{\cdot}_\Gamma$ is the jump over $\Gamma$.
The density $\vr$ will be approximated in the space 
of piecewise constants on $E_h$:
\begin{equation*}
	Q_h(\Om) = \Set{q \in L^2(\Om):~q|_E \in \mathbb{P}_0^1(E), ~\forall E \in E_h}.
\end{equation*}

Next, we introduce the canonical interpolation operators: 
\begin{equation*}
	\begin{split}
	&\Pi_{h}^S:W^{1,2}_{0}\cap ~W^{2,2} \rightarrow S_{h}, \quad 
	\Pi_{h}^W:\vc{W}^{\text{curl},2}_0\cap ~\vc{W}^{2,2} \rightarrow \vc{W}_{h}, \\
	&\Pi_{h}^V:\vc{W}^{\text{div},2}_0\cap \vc{W}^{1,2} \rightarrow \vc{V}_{h}, \quad 
	\Pi_{h}^Q: L^2_{0} \rightarrow Q_{h},	
\end{split}
\end{equation*}
using the available degrees of freedom of the involved spaces. That is,
the operators (in three dimensions) are defined by \cite{Nedelec:1980ec,Arnold} 
\begin{equation*}
	\begin{split}
		\left(\Pi_h^S s\right)(x_i) &= s(x_i), \quad \forall x_i \in \mathcal{N}_h; \\
		\int_e \left(\Pi_h^W \vc{w}\right)\times \nu~dS(x) 
			&= \int_e  \vc{w}\times \nu~dS(x), 
				\quad \forall e \in \mathcal{V}_h;\\
		\int_\Gamma \left(\Pi_h^V \vc{v}\right)\cdot \nu~dS(x) &= \int_\Gamma \vc{v} \cdot \nu~dS(x), 
				\quad \forall \Gamma \in \Gamma_h;\\
		\int_E \Pi_h^Q q~dx &= \int_E q ~dx, \quad \forall E \in E_h,
	\end{split}
\end{equation*}
where $\mathcal{N}_h$ it the set of vertices of $E_h$.
It is well known that the following diagram commutes (\cite{Arnold, Arnold:2006wj}):
\begin{equation*}
	{\small 
	\begin{CD}
		 W^{1,2}_{0}\cap W^{2,2} @> 
		\operatorname{grad} >> \vc{W}^{\text{curl},2}_{0}\cap ~\vc{W}^{2,2} @> \operatorname{curl}\ >> 
		\vc{W}^{\text{div}, p}_{0}\cap \vc{W}^{1,2} @> \operatorname{div}\ >> \vc{L}^2_{0}\\
		 @V \Pi_{h}^SVV    @V\Pi_{h}^WVV   @V\Pi_{h}^{V}VV          
		@V\Pi_{h}^QVV  \\ 
		S_{h} @>\operatorname{grad} >> \vc{W}_{h} @> \operatorname{curl}\ >> 
		\vc{V}_{h} @> \operatorname{div}\ >> Q_{h}.
	\end{CD}
	}
\end{equation*}

\begin{remark}
	The interpolation operators $\Pi_h^S$, $\Pi_h^W$, and $\Pi_h^V$, are 
	defined on function spaces with enough regularity to ensure that the 
	corresponding degrees of freedom are functionals on these spaces. 
	This is reflected in writing $\vc{W}^{\text{curl},2} \cap \vc{W}^{2,2}$
	instead of merely $\vc{W}^{\text{curl},2}$ and so on.
\end{remark}
In view of the above commuting diagram, we can define the spaces orthogonal to the range of the previous operator, i.e., 
\begin{align*}
	\vc{W}_{h}^{0,\perp} & := \{\vc{w}_{h} \in \vc{W}_{h}; \Curl \vc{w}_{h} =0 \}^\perp\cap \vc{W}_{h}, \\
	\vc{V}_{h}^{0,\perp} & := \{\vc{v}_{h} \in \vc{V}_{h}; \Div \vc{v}_{h} = 0\}^\perp \cap \vc{V}_{h},
\end{align*}
to obtain decompositions (cf. \cite{Arnold})
\begin{align}
	\vc{W}_{h} &= \Grad S_{h} +\vc{W}_{h}^{0, \perp}, \label{eq:Wh-decomp}
	\\
	\vc{V}_{h}  & = \Curl \vc{W}_{h} +\vc{V}_{h}^{0, \perp}. \label{eq:Vh-decomp}
\end{align}
The following discrete Poincar\'e inequalities hold \cite{Arnold:2006wj}
\begin{align}
	\label{eq:Poincare1}
	\norm{\vc{v}_{h}}_{\vc{L}^2(\Omega)} & \leq 
	C\norm{\Div \vc{v}_{h}}_{L^2(\Omega)},
	\quad \forall \vc{v} \in \vc{V}_{h}^{0,\perp},\\
	\label{eq:Poincare2}
	\norm{\vc{w}_{h}}_{\vc{L}^2(\Omega)} 
	& \leq C\norm{\Curl \vc{w}_{h}}_{L^2(\Omega)},
	\quad \forall \vc{w} \in \vc{W}_{h}^{0,\perp},
\end{align}	
where the constant $C$ is independent of $h$.

In the subsequent convergence analysis, we make frequent 
use of the canonical projection operators. To 
bound these we shall need the following (\cite{Nedelec:1980ec, Brenner})

\begin{lemma}\label{lemma:interpolation}
There exists a constant $C>0$, depending only on the shape 
regularity of $E_h$ and the size of $\Om$, such that for any $1\leq p\leq\infty$,
\begin{align*}
	& \norm{\phi - \Pi_h^Q\phi}_{L^p(\Om)} \leq C h\norm{\Grad \phi}_{\vc{L}^p(\Om)}, \\
	& \norm{\vc{v} - \Pi_h^V \vc{v}}_{\vc{L}^p(\Om)}
	+h\norm{\Div (\vc{v} - \Pi_h^V \vc{v})}_{L^p(\Om)}
	\leq C h^{s}\norm{\Grad ^{s}\vc{v}}_{\vc{L}^p(\Om)}, \quad r=1,2, \\
	& \norm{\vc{w} - \Pi_h^W \vc{w}}_{\vc{L}^p(\Om)}
	+ h\norm{\Curl (\vc{w} - \Pi_h^W \vc{w})}_{\vc{L}^p(\Om)} 
	\leq C h^{s}\|\Grad ^{s}\vc{w}\|_{\vc{L}^p(\Om)}, \quad s=1,2,
\end{align*}
for all $\phi \in W^{1,p}(\Om),\vc{v}\in W^{s,p}(\Om)$, 
and $\vc{w} \in W^{2,p}(\Om)$.
\end{lemma}

We will also need the following lemma. 
It follows from scaling arguments and the equivalence of 
finite dimensional norms \cite{Brenner}.
\begin{lemma}\label{lemma:inverse}
There exists a constant $C>0$, such that for $1\leq q,p \leq \infty$,
and $r= 0,1$,
\begin{equation*}
	\norm{\phi_h}_{W^{r,p}(E)} 
	\leq C h^{-r + \frac{N}{p}-\frac{N}{q}}
	\norm{\phi_h}_{L^q(E)}, 
\end{equation*}
for any $E \in E_h$ and all polynomial functions $\phi_h \in \mathbb{P}_k(E)$, $k=0,1,.$.
The constant $C$ depends only on the shape 
regularity of $E_h$ and polynomial degree $k$.
\end{lemma}

The next result follows from scaling arguments and the trace theorem \cite{Agmonn}
\begin{lemma}\label{lemma:edgebounds}
Fix any $E \in E_{h}$ and let $\phi \in W^{1,2}(E)$ be arbitrary. 
There exists a constant $C>0$, depending only on the shape regularity of $E_h$ such that, 
$$
\|\phi\|_{L^2(\Gamma)} \leq Ch^{-\frac{1}{2}}\left(\|\phi\|_{L^2(E)} + h\|\Grad  \phi\|_{\vc{L}^2(E)}\right), \quad \forall \Gamma \in \Gamma_{h}\cap \partial E.
$$
\end{lemma}

% To simplify notation,
% we define the operator $\Aoph{\cdot}: Q_{h}(\Om)\cap L^p_{0}(\Om) \rightarrow \vc{V}_{h}(\Om)$ 
% by
% \begin{equation}\label{eq:aoph}
% \Aoph{\phi_{h}} = \Pi_{h}^V \Aop{\phi_{h}}, \quad \phi_{h} \in Q_{h}(\Om)\cap L^p_{0}(\Om),
% \end{equation}
% where $\Aop{\cdot}$ is given by \eqref{eq:aop}.

We now establish a Sobolev embedding estimate 
for the discrete decompositions \eqref{eq:Vh-decomp} and \eqref{eq:Wh-decomp}. 

\begin{lemma}
 \label{lemma:embedding}
The finite element spaces $\vc{V}_{h}^{0,\perp}(\Om)$ and $\vc{W}_{h}^{0,\perp}(\Om)$ satisfies the
following embedding results independent of $h$:
\begin{enumerate}
\item{}The space $\vc{V}_{h}^{0,\perp}(\Om)\cap \vc{W}_{0}^{\operatorname{div}, 2}(\Omega)$ is embedded in $\vc{L}^{2^*}(\Om)$,
\item{}The space $\vc{W}_{h}^{0,\perp}(\Om)\cap \vc{W}_{0}^{\operatorname{curl}, 2}(\Om)$ is embedded in $\vc{L}^{2^*}(\Om)$,
\end{enumerate}
where $2^* = 6$ if $N=3$, and $2^*$ is any large finite number if $N=2$.
\end{lemma}

\begin{proof}
We first prove (1).
By virtue of the decomposition \eqref{eq:Vh-decomp} we can for any 
$\vc{v}_{h} \inb \vc{V}_{h}^{0,\perp}(\Om)\cap \Hdiv $ find
functions $\vc{\zeta}_{h} \in \vc{W}_{h}(\Omega)$ and $\vc{z}_{h} \in \vc{V}_{h}^{0,\perp}(\Omega)$
such that
$$
	\Aoph{\Div \vc{v}_{h}} = \Curl \vc{\xi}_{h} + \vc{z}_{h}.
$$
Using the commutative diagram and the definition of $\Aoph{\cdot}$ we easily verify that
$$
\Div \Aoph{\Div \vc{v}_{h}} = \Div \vc{v}_{h}.
$$
Hence, since $(\vc{z}_{h} - \vc{v}_{h}) \in \vc{V}_{h}^{0,\perp}(\Omega)$ we can use the discrete Poincar\'e inequality \eqref{eq:Poincare1}
to conclude that
\begin{equation*}
\|\vc{v}_{h} - \vc{z}_{h}\|_{\vc{L}^2(\Omega)} \leq C\|\Div (\vc{v}_{h} - \vc{z}_{h})\|_{L^2(\Omega)} = 0.
\end{equation*}
Thus, $\vc{z}_{h} = \vc{v}_{h}$ a.e in $\Omega$ and we easily calculate
\begin{equation*}
\begin{split}
\|\vc{v}_{h}\|_{\vc{L}^{2^*}(\Omega)} &= \|\vc{z}_{h}\|_{\vc{L}^{2^*}(\Omega)} 
\leq \|\Aoph{\Div \vc{v}_{h}}\|_{\vc{L}^{2^*}(\Om)} \\
&\leq C_{1}\|\Aop{\Div \vc{v}_{h}}\|_{\vc{L}^{2^*}(\Om)} \leq C_{2}\|\Div \vc{v}_{h}\|_{L^2(\Om)},
\end{split}
\end{equation*}
where the last inequality is the standard Sobolev embedding $\vc{W}^{1,2}(\Om) \subset \vc{L}^{2^*}(\Om)$.

In two spatial dimensions, (2) follows directly from the standard Sobolev embedding $W^{1,2}(\Om) \subset L^{2^*}(\Om)$.
To prove (2) in three spatial dimensions, fix any $\vc{w}_{h} \inb \vc{W}_{h}^{0,\perp}(\Om)\cap \Hcurl$ and 
let $\vc{\eta} \in \Hcurl\cap \vc{W}^{\text{div}, 2}(\Om) \subset \vc{W}^{1,2}(\Om)$ solve (cf. \cite{Girault:1986fu})
\begin{equation*}
\begin{split}
\Curl \vc{\eta} &= \Curl \vc{w}_{h}, \textrm{ in }\Omega, \\
\Div \vc{\eta} &= 0, \textrm{ in } \Omega, \\
\vc{\eta} \times \nu &= 0, \textrm{ on }\partial \Omega.
\end{split}
\end{equation*}
Using the decomposition \eqref{eq:Wh-decomp} of the space $\vc{W}_{h}(\Om)$, we can find functions
$s_{h} \in S_{h}(\Om)$ and $\vc{\zeta}_{h} \in \vc{W}_{h}^{0,\perp}(\Om)$ such that
$$
	\Pi_{h}^W\vc{\eta} = \Grad s_{h} + \vc{\zeta}_{h}.
$$
Hence, from the commuting diagram property, we deduce
$$
	\Curl \vc{\zeta}_{h}= \Curl \Pi_{h}^W \vc{\eta} = \Pi_{h}^V \Curl \vc{\eta} = \Pi_{h}^V \Curl \vc{w}_{h} = \Curl \vc{w}_{h}.
$$
Thus, since $(\vc{w}_{h} - \vc{\zeta}_{h}) \in \vc{W}_{h}^{0,\perp}(\Om)$ we can use the Poincar\'e inequality \eqref{eq:Poincare2}
to conclude that 
$$
	\|\vc{w}_{h} - \vc{\zeta}_{h}\|_{\vc{L}^2(\Om)} \leq C\|\Curl (\vc{w}_{h} - \vc{\zeta}_{h})\|_{\vc{L}^2(\Om)}= 0,
$$
and hence that $\vc{w}_{h} = \vc{\zeta}_{h}$. Moreover, we easily calculate
$$
\|\vc{w}_{h}\|_{\vc{L}^{2^*}(\Om)} = \|\vc{\zeta}_{h}\|_{\vc{L}^{2^*}(\Om)} \leq \|\Pi_{h}^W \vc{\eta}\|_{\vc{L}^{2^*}(\Om)} 
\leq C_{1}\|\vc{\eta}\|_{\vc{L}^{2^*}(\Om)} \leq C_{2}\|\Grad \vc{\eta}\|_{\vc{L}^{2}(\Om)},
$$
where the last inequality is the standard Sobolev embedding $\vc{W}^{1,2}(\Om) \subset \vc{L}^2(\Om)$.
This concludes the proof.
\end{proof}

We end this section by recalling a compactness result from \cite[Theorem A.1]{Karlsen1}.

\begin{lemma}\label{lemma:spacetranslation}
Let $\{\vc{v}_{h}\}_{h>0}$ be a sequence of functions in $\vc{V}^{0,\perp}_{h}$ with
$\Div \vc{v}_h \inb L^2(\Om)$.
For any $\xi \in \mathbb{R}^N$,
\begin{equation*}
\|\vc{v}_{h}(x) - \vc{v}_{h}(x- \xi) \|_{\vc{L}^2(\Omega)} \leq C(|\xi|^{\frac{4-N}{2}} + |\xi|^2)^\frac{1}{2}\|\Div \vc{v}_{h}\|_{L^2(\Omega)},
\end{equation*}
where the constant $C>0$ is independent of both $h$ and $\xi$.
\end{lemma}

%\begin{lemma}\label{lemma:productconv}
%If $\{f^m_{h}\}_{h>0}$ and $\{g_{h}^m\}_{h>0}$ is such that $f_{h} \weak f$ in $L^2(0,T;L^2(\Omega))$
%, $g_{h} \weak g$ in $L^2(0,T;L^2(\Omega))$,
%$$
%	\Pth{g_{h}} \inb L^1(0,T;W^{-m,1}(\Omega)),\quad m>0,
%$$
%and
%$$
%	\|f_{h}(t,x)  - f_{h}(t,x-\xi)\|_{L^2(0,T;L^2(\Omega))} \rightarrow 0, \textrm{ as }|\xi| \rightarrow 0,
%$$
%uniformly in $h$. Then,
%$$
%	g_{h}f_{h} \weak gf,
%$$
%in the sense of distributions on $(0,T)\times \Omega$.
%\end{lemma}

%\begin{proof}
%The proof is identical to that of Lemma in \cite{}.
%\end{proof}

\section{Numerical method and main result}\label{sec:numerical-method}
In this section we define the numerical method for the Stokes approximation equations
 and the state the main convergence theorem. The proof 
of the main theorem is deferred to subsequent sections.

%:Current position - mandag; 26. mai, 2008, 13:29

Given a time step $\Dt>0$, we discretize the time interval $[0,T]$ in 
terms of the points $t^m=m\Dt$, $m=0,\dots,M$, where we assume that $M\Dt=T$.  
Regarding the spatial discretization, we let $\{E_{h}\}_{h}$ be a shape regular 
family of tetrahedral meshes of $\Omega$,  where $h$ is the maximal diameter.  
It will be a standing assumption that $h$ and $\Delta t$ are 
related such that $\Delta t = c h$, for some constant $c$. 
Furthermore, for each $h$, let $\Gamma_{h}$ denote the set of faces in $E_{h}$. 

% On each element $E \in E_{h}$, we denote by $Q(E)$ the constants 
% on $E$. The functions that are piecewise constant with respect to the elements of a mesh 
% $E_{h}$ are denoted by $Q_h=Q_h(\Om)$.  
% Next, on each $E \in E_{h}$, we let $\vc{W}(E)$ be the lowest 
% order space of \emph{curl--conforming N\'ed\'elec} 
% polynomials of first kind \cite{Nedelec:1980ec}. 
% In two dimensions, $\vc{W}(E)$ is the space of linear scalar polynomials on $E$. 
% On each element $E \in E_{h}$, we let $\vc{V}(E)$ be the 
% lowest order space of \emph{div--conforming N\'ed\'elec} 
% polynomials of first kind \cite{Nedelec:1980ec}. In two dimensions, it is 
% the Raviart--Thomas polynomial space on $E$. 
% The element spaces $\vc{W}_{h}=\vc{W}_{h}(\Omega)$ and $\vc{V}_{h}=\vc{V}_{h}(\Omega)$ 
% are formed on the entire mesh $E_{h}$ by matching the 
% degrees of freedom of the polynomial space $\vc{W}(E)$ and $\vc{V}(E)$, respectively,  
% on each face $\Gamma \in \Gamma_{h}$. 
% In addition, we incorporate the boundary conditions by letting the degrees of freedom of the 
% spaces $\vc{W}_{h}$ and $\vc{V}_{h}$ vanish at the faces on the boundary.
For each fixed $h>0$, we let $\vc{W}_h(\Om)$ and $\vc{V}_h(\Om)$
denote the N\'ed\'elec  spaces of the first order and kind
on $E_h$ (cf. Section 2.3) and $Q_h(\Om)$ 
the space of piecewise constants on $E_h$. To incorporate boundary conditions, we let the degrees 
of freedom of $\vc{W}_h(\Om)$ and $\vc{V}_h(\Om)$ located at
the boundary $\pOm$ vanish.

Before defining our numerical scheme, we shall need to introduce some 
additional notation related to the discontinuous Galerkin scheme. 
Concerning the boundary $\partial E$ of an element $E$, we write $f_{+}$ 
for the trace of the function $f$ achieved from within the element $E$ 
and $f_{-}$ for the trace of $f$ achieved from outside $E$. 
Concerning an edge $\Gamma$ that
is shared between two elements $E_{-}$ and $E_{+}$, we will write $f_{+}$ for 
the trace of $f$ achieved from within $E_{+}$ and $f_{-}$ for the trace
of $f$ achieved from within $E_{-}$. Here $E_{-}$ and $E_{+}$ are 
defined such that $\nu$ points from $E_{-}$ to $E_{+}$, where $\nu$ is 
fixed (throughout) as one of the two possible 
normal components on each edge $\Gamma$ throughout the discretization.
We also write $\jump{f}_{\Gamma}= f_{+} - f_{-}$ for the jump of $f$ 
across the edge $\Gamma$, while forward time-differencing 
of $f$ is denoted by $\jump{f^m} = f^{m+1} - f^m$.
Discrete implicit time discretization of a function $f$ is
denoted by the operator $\Pth{f^m} = \frac{1}{\Delta t}\jump{f^{m-1}}$.

Let us now define our numerical scheme. % for the semi-stationary 
% Stokes system \eqref{eq:contequation}--\eqref{eq:momentumeq} 
% augmented with the boundary conditions \eqref{eq:bc-normal} and \eqref{eq:bc-navierslip} (note, however, 
% that in the definition below the boundary conditions are built into the 
% finite element spaces and not listed explicitly).

\begin{definition}[Numerical scheme]\label{def:num-scheme}
Let $\Set{\vrho^0_h(x)}_{h>0}$ be a sequence (of piecewise constant 
functions) in $Q_{h}(\Omega)$ that satisfies $\vrho_h^0>0$ for each fixed $h>0$ 
and $\varrho^0_h\to \vrho^0$ a.e.~in $\Om$ and in $L^1(\Om)$ 
as $h\to 0$. Let the sequence $\{\vu_h^0\}_{h>0}$ be
such that for each fixed $h>0$,  $\vu_h^0 \in \vc{V}_h(\Om)$ and satisfies
\begin{equation}\label{eq:l2proj}
\int_\Om \vu_h^0 \vc{v}_h~dx = \int_\Om \vu^0 \vc{v}_h~dx, \quad \forall \vc{v}_h \in \vc{V}_h(\Om). 	
\end{equation}

Now, determine functions 
$$
(\varrho^m_{h},\vc{w}^m_{h},\vc{u}^m_{h}) \in Q_{h}(\Omega)\times\vc{W}_{h}(\Omega)
\times \vc{V}_{h}(\Omega), \quad m=1,\dots,M,
$$
such that for all $\phi_{h} \in Q_{h}(\Omega)$,
\begin{equation}\label{FEM:contequation}
	\begin{split}
		&\int_\Omega \Pth{\varrho^m_h} \phi_{h}\ dx
		=\Delta t\sum_{\Gamma \in \Gamma^I_h}\int_\Gamma \left(\varrho^m_{-}(\vc{u}^{m}_h \cdot \nu)^+
		+\varrho^m_+(\vc{u}^{m}_h \cdot \nu)^-\right)\jump{\phi_{h}}_\Gamma\ dS(x) ,
	\end{split}
\end{equation}
and for all $(\vc{\eta}_{h},\vc{v}_{h}) \in \vc{W}_{h}(\Omega)\times \vc{V}_{h}(\Omega)$,
\begin{equation}\label{FEM:momentumeq}
	\begin{split}
		&\int_{\Omega}\Pth{\vc{u}_{h}^m}\vc{v}_{h} + \mu\Curl \vc{w}^m_{h}\vc{v}_{h} + \left[(\mu + \lambda)\Div \vc{u}^m_{h}
		-p(\varrho^m_{h})\right]\Div \vc{v}_{h}\ dx
		= 0, \\
		&\int_{\Omega}\vc{w}^m_{h}\vc{\eta}_{h} -  \vc{u}^m_{h}\Curl \vc{\eta}_{h} \ dx =0,
	\end{split}
\end{equation}
for $m=1,\dots,M$. 

In \eqref{FEM:contequation}, $(\vc{u}_{h} \cdot \nu)^+=\max\{\vc{u}_{h} \cdot \nu, 0\}$ 
and $(\vc{u}_{h} \cdot \nu)^+ = \min\{\vc{u}_{h} \cdot \nu, 0\}$, so that 
$\vc{u}_{h} \cdot \nu=(\vc{u}_{h} \cdot \nu)^++(\vc{u}_{h} \cdot \nu)^-$, i.e., in the 
evaluation of $\varrho(\vc{u} \cdot \nu)$ at the edge $\Gamma$ the 
trace of $\varrho$ is taken in the upwind direction. 
\end{definition}

\begin{remark}\label{rem:E-VS-Gamma}
Recall that $\vrho_\pm$ and $(\vc{u}_h \cdot \nu)^\pm$ related to 
a face $\Gamma$ has a different meaning than  $\vrho_\pm$ and $(\vc{u}_h \cdot \nu)^\pm$ 
related to the boundary of an element $\partial E$.
By direct calculation, one can verify the identity
\begin{align*}
	&\Delta t\sum_{E \in E_{h}}\int_{\binner}\left(\vrho^m_{+}(\vc{u}_{h}^{m} \cdot \nu)^+
	+\vrho^m_{-}(\vc{u}_{h}^{m} \cdot \nu)^-\right)\phi_{h} \ dS(x) \\
	&\qquad 
	= -\Delta t\sum_{\Gamma \in \Gamma^I_{h}}\int_{\Gamma}\left(\vrho^m_{-}(\vc{u}_{h}^{m} \cdot \nu)^+
	+\vrho^m_{+}(\vc{u}_{h}^{m} \cdot \nu)^-\right)[\phi_{h}]_\Gamma \ dS(x).
\end{align*}
Using this identity, we can state \eqref{FEM:contequation} on the following form:
\begin{equation}\label{FEM:contequation-newform}
	\begin{split}
		&\int_\Omega \vrho^m_h \phi_{h}\ dx
		+\Delta t\sum_{E \in E_{h}}\int_{\binner}\left(\vrho^m_{+}(\vc{u}_{h}^{m} \cdot \nu)^+
		+\vrho^m_{-}(\vc{u}_{h}^{m} \cdot \nu)^-\right)\phi_{h} \ dS(x)
		\\ & \qquad = \int_\Omega \vrho^{m-1}_h\phi_{h}\ dx.
	\end{split}
\end{equation}
The form \eqref{FEM:contequation-newform} will be used frequently  in the subsequent analysis.

\end{remark}

For each fixed $h>0$, the numerical solution 
$\Set{(\varrho^m_{h},\vc{w}^m_{h},\vc{u}^m_{h})}_{m=0}^M$
is extended to the whole of $(0,T)\times \Omega$ by setting 
\begin{equation}\label{eq:num-scheme-II}
	(\varrho_{h},\vc{w}_{h},\vc{u}_{h})(t)=(\varrho^m_{h},\vc{w}^m_{h},\vc{u}^m_{h}), 
	\qquad t\in (t_{m-1},t_m), \quad m=1,\dots,M.
\end{equation}
In addition, we set $\varrho_{h}(0)= \varrho^0_{h}$ and $\vu_h(0)= \vu_h^0$. 

The continuity scheme \eqref{FEM:contequation} clearly preserves the total
mass. The following lemma  from \cite[Lemma 4.1]{Karlsen1} states
that the density is strictly positive whenever the initial density is strictly positive.
\begin{lemma} \label{lemma:vrho-props}
Fix any $m=1,\dots,M$ and suppose $\varrho^{m-1}_{h} 
\in Q_{h}(\Om)$, $\vc{u}^m_{h} \in \vc{V}_{h}(\Om)$ 
are given bounded functions. Then the solution $\varrho^{m}_{h} \in Q_{h}(\Om)$ of 
the discontinuous Galerkin scheme \eqref{FEM:contequation} satisfies
$$
\min_{x \in \Omega}\varrho_{h}^m \geq \min_{x \in \Omega}\varrho_{h}^{m-1}
\left(\frac{1}{1 + \Delta t \|\Div \vc{u}^m_{h}\|_{L^\infty(\Omega)}}\right).
$$
Consequently, if $\varrho^{m-1}_{h}(\cdot)>0$, then $\varrho^{m}_{h}(\cdot)>0$.
\end{lemma}

Existence of a solution to the nonlinear--implicit  discrete scheme 
follows from a topological degree argument.
This argument is essentially identical to that of \cite[Lemma 4.2]{Karlsen1}
with a minor modification to accommodate the discrete time derivative $\Pth{\vu_h}$.

\begin{lemma}
%Given a time $T > 0$, there exists, 
For each fixed $h > 0$, there exists a solution 
$$
(\varrho^m_{h},\vc{w}^m_{h},\vc{u}^m_{h}) 
\in Q_{h}(\Omega)\times\vc{W}_{h}(\Omega)
\times \vc{V}_{h}(\Omega), \quad 
\vrho^m_h(\cdot)>0, \quad m=1,\dots,M,
$$
to the nonlinear--implicit discrete problem 
posed in Definition \ref {def:num-scheme}. 
\end{lemma}

Our main result is that, passing if necessary to a subsequence, 
$\Set{(\varrho_{h},\vc{w}_{h},\vc{u}_{h})}_{h>0}$ 
converges to a weak solution. More precisely, there holds

\begin{theorem}[Convergence]\label{theorem:mainconvergence}
Suppose $(\vrho^0, \vu^0)\in L^\gamma(\Om)\cap \vc{L}^2(\Om)$, $\gamma > \frac{N}{2}$.
Let $\Set{(\varrho_{h},\vc{w}_{h},\vc{u}_{h})}_{h>0}$ be a sequence of numerical solutions 
constructed according to \eqref{eq:num-scheme-II} and Definition \ref {def:num-scheme}. 
Then, passing if necessary to a subsequence as $h\to 0$, 
$\vu_h \rightarrow \vu$, a.e in $\Dom$,
$\varrho_{h}\vc{u}_{h} \weak \varrho\vc{u}$ in the sense 
of distributions on $\Dom$, and $\varrho_{h} \rightarrow \varrho$
a.e.~in $\Dom$, where the limit triplet $(\vrho,\vc{w}, \vc{u})$ satisfies the mixed formulation
\eqref{def:mixed-weak}, and thus $(\vrho,\vc{u})$ is a weak 
solution according to Definition \ref{def:weak}.
\end{theorem}

\section{Basic estimates}\label{sec:basic-est}
In this section we gather some basic 
estimates for our numerical method.
The results include stability and weak time-continuity 
of both the density and velocity. 
We however commence by recalling (from \cite{Karlsen1})
the following renormalized version of the continuity scheme.

\begin{lemma}[Renormalized continuity scheme]
Fix any $m=1,\ldots,M$ and let the pair $(\vrho_{h}^m, \vc{u}_{h}^{m}) \in Q_{h} \times \vc{V}_{h}$ satisfy the continuity scheme \eqref{FEM:contequation}. 
Then $(\vrho_{h}^m, \vc{u}_{h}^{m})$ also satisfies 
the renormalized continuity scheme
\begin{equation}\label{FEM:renormalized}
	\begin{split}
		&\int_{\Omega}B(\vrho_{h}^m)\phi_{h}\ dx \\
		&\qquad - \Delta t\sum_{\Gamma \in \Gamma^I_{h}} \int_\Gamma \left(B(\vrho^m_-)(\vc{u}^{m}_h \cdot \nu)^+ 
		+ B(\vrho^m_+)(\vc{u}^{m}_h \cdot \nu)^-\right) \jump{\phi_{h}}_\Gamma\ dx \\
		&\qquad + \Delta t \int_{\Omega}b(\vrho_{h}^m)\Div \vc{u}^{m}_{h}\phi_{h}\ dx 
		+  \int_{\Omega}B''(\xi(\vrho_{h}^m, \vrho_{h}^{m-1})) \jump{\vrho_{h}^{m-1}}^2\phi_{h}\ dx  \\
		&\qquad  
		+\Delta t\sum_{\Gamma \in \Gamma^I_{h}}
		\int_{\Gamma}B''(\xi^\Gamma(\vrho^m_{+},\vrho^m_{-}))
		\jump{\vrho^m_{h}}^2_{\Gamma}(\phi_{h})_{-}(\vc{u}_{h}^{m} \cdot \nu)^+ \\
		&\qquad \qquad \qquad \qquad -B''(\xi^\Gamma(\vrho^m_{-},
		\vrho^m_{+}))\jump{\vrho^m_{h}}^2_{\Gamma}(\phi_{h})_{+}(\vc{u}_{h}^{m} \cdot \nu)^-\ dS(x) \\
		&= \int_{\Omega}B(\vrho_{h}^{m-1})\phi_{h}\ dx,
		\qquad \forall \phi_{h} \in Q_{h}(\Omega),
	\end{split}
\end{equation} 
for any $B\in C[0,\infty)\cap C^2(0,\infty)$ with $B(0)=0$ 
and $b(\vrho) := \vrho B'(\vrho) - B(\vrho)$. 
Given two positive real numbers $a_1$ and $a_2$, we denote by $\xi(a_1,a_2)$ and $\xi^\Gamma(a_1,a_2)$ two numbers 
between $a_1$ and $a_2$ (See \cite{Karlsen1} for a precise definition).
\end{lemma}

In what follows we will need the following discrete Hodge decomposition.

\begin{lemma}\label{lemma:hodge}
\solutiontext For each fixed $h>0$, there exist unique 
functions $\vc{\zeta}_{h}^m \in \vc{W}_{h}^{0,\perp}$ and 
$\vc{z}_{h}^m \in \vc{V}_{h}^{0,\perp}$ such that
\begin{equation}\label{eq:hodge-disc-time}
	\vc{u}_{h}^m = \Curl \vc{\zeta}_{h}^m 
	+ \vc{z}_{h}^m, \qquad m=0,\ldots,M.
\end{equation}
Moreover, if we let $\vc{\zeta}_{h}(t,x)$, $\vc{z}_{h}(t,x)$ denote 
the functions obtained by extending, as in \eqref{eq:num-scheme-II}, 
$\{\vc{\zeta}_{h}^m\}_{m=1}^M$, $\{\vc{z}_{h}^m\}_{m=1}^M$ 
to the whole of $(0,T] \times \Omega$, then
\begin{equation*}%\label{eq:hodge-cont-time}
	\vc{u}_{h}(t,\cdot) = \Curl \vc{\zeta}_{h}(\cdot,t)
	+\vc{z}_{h}(\cdot,t), \qquad t \in (0,T).
\end{equation*}

Finally, let $\Curl \vc{\zeta}^0 \in \vc{L}^2(\Om)$ and $\Grad s^0 \in \vc{L}^2(\Om)$ satisfy
the standard continuous Hodge decomposition 
$\vu^0 = \Curl \vc{\zeta}^0 + \Grad s^0$. Then,
$$
	\Curl \vc{\zeta}^0_h \rightarrow \Curl \vc{\zeta}^0,\quad \vc{z}_h^0 \rightarrow \Grad s^0,\quad \text{in }\vc{L}^2(\Om),
$$
where $\vc{\zeta}_h^0$ and $\vc{z}_h^0$ are given by \eqref{eq:hodge-disc-time}.

\end{lemma}
\begin{proof}
The  first two statements are consequences of \eqref{eq:Vh-decomp}.

To prove the last statement, fix any $\vc{\phi} \in \vc{C}_c^\infty(\Om)$ and set $\vc{v}_h = \Pi_h^W \vc{\phi}$ 
in \eqref{eq:l2proj} to obtain
\begin{equation}\label{eq:init1}
	\int_\Om \Curl \vc{\zeta}^0_h \Curl( \Pi_h^W\vc{\phi})~dx
	= \int_\Om \Curl \vc{\zeta}^0 \Curl( \Pi_h^W\vc{\phi})~dx,
\end{equation}
where we have used that $\vc{u}_h^0 = \Curl \vc{\zeta}_h^0 + \vc{z}_h^0$ and $\int_\Om \vc{z}_h^0 \Curl \Pi_h^W \vc{\phi}~dx = 0$.
Now, since $\|\Curl \vc{\zeta}_h^0\|_{\vc{L}^2(\Om)} \leq C\|\vu_h^0\|_{\vc{L}^2(\Om)} \leq C\|\vu^0\|_{\vc{L}^2(\Om)}$, 
there exists a function $\overline{\Curl \vc{\zeta}^0}$ such that
$\Curl \vc{\zeta}_h^0 \weak \overline{\Curl \vc{\zeta}^0}$ in $\vc{L}^2(\Om)$. 
Sending $h \rightarrow 0$ in \eqref{eq:init1} yields
\begin{equation*}
	\int_\Om (\overline{\Curl \vc{\zeta}^0} - \Curl \vc{\zeta}^0)\Curl \vc{\phi}~dx = 0, \quad \forall \vc{\phi} \in \vc{C}_c^\infty(\Om).
\end{equation*}
Hence, $\overline{\Curl \vc{\zeta}^0} = \Curl \vc{\zeta}^0$ a.e in $\Om$.

Next, let $\vc{v}_h = \Curl \vc{\zeta}_h$ in \eqref{eq:l2proj} to discover
\begin{equation*}%\label{eq:stinit}
	\|\Curl \vc{\zeta}_h^0\|_{\vc{L}^2(\Om)}^2 = \int_\Om \Curl \vc{\zeta}^0 \Curl \vc{\zeta}_h^0 ~dx
	\rightarrow \|\Curl \vc{\zeta}^0\|_{\vc{L}^2(\Om)}^2,
\end{equation*}
as $h \rightarrow 0$. Then, $\Curl \vc{\zeta}_h^0 \rightarrow \Curl \vc{\zeta}^0$ in $\vc{L}^2(\Om)$.

By setting $\vc{v}_h = \vu_h^0$ in \eqref{eq:l2proj} we deduce
\begin{equation*}
	\|\vu_h^0\|_{\vc{L}^2(\Om)}^2 = \int_\Om \vu^0 \vu_h^0~dx \rightarrow \|\vu^0\|_{\vc{L}^2(\Om)}^2,
\end{equation*}
as $h \rightarrow 0$. Hence, $\vu_h^0 \rightarrow \vu^0$ in $\vc{L}^2(\Om)$. 

Finally, a direct calculation shows that
\begin{equation*}
	\begin{split}
		0 = \lim_{h \rightarrow 0}\|\vc{u}_h^0 - \vu^0\|_{\vc{L}^2(\Om)}^2 
		& = \lim_{h \rightarrow 0}\left[\|\Curl \vc{\zeta}_h^0 - \Curl \vc{\zeta}^0\|_{\vc{L}^2(\Om)}^2 + \|\vc{z}_h^0 - \Grad s^0\|_{\vc{L}^2(\Om)}^2\right] \\
		& \qquad - 2\lim_{h \rightarrow 0}\left[\int_\Om (\Curl \vc{\zeta}_h^0 - \Curl \vc{\zeta}^0)\,(\vc{z}_h^0 - \Grad s^0)~dx\right],
	\end{split}
\end{equation*}
where the last term converges to zero 
since $\Curl \vc{\zeta}_h^0 \to \Curl \vc{\zeta}^0$
in $\vc{L}^2(\Om)$. Thus, $\vc{z}_h^0 \rightarrow \Grad s^0$
in $\vc{L}^2(\Om)$ and the proof is complete.

\end{proof}

We now derive a basic stability estimate satisfied 
by the numerical scheme.

\begin{lemma}[Stability]\label{lemma:stability} \solutiontext 
For $\vr(\cdot)>0$, let
\begin{equation*}%\label{eq:energydef}
	\mathcal{E}(\varrho, \vc{u})  = \frac{a}{\gamma-1}\vr^\gamma + \frac{1}{2}|\vu|^2.
\end{equation*}
For any $m=1,\dots,M$, there holds
\begin{equation*}
	\begin{split}
		&\int_{\Omega} \mathcal{E}(\varrho_{h}^m, \vc{u}_{h}^m)\ dx  
		 +\sum_{k=1}^m \Delta t\|\vc{u}^k_{h}\|_{\vc{W}^{\textnormal{div}, 2}(\Om)}^2 
			+ \sum_{k=1}^m \Delta t\|\vc{w}_{h}^k\|_{\vc{W}^{\textnormal{curl},2}(\Om)}^2 
			+ \mathcal{N}_\text{diffusion}^m \\
		&\qquad	 \leq \int_{\Omega} \mathcal{E}(\vrho^0, \vc{u}^{0}) \ dx,
	\end{split}
\end{equation*}
where the numerical diffusion term is given by
\begin{equation*}
	\begin{split}
		\mathcal{N}_\text{diffusion}^m 
		&= \frac{1}{2}\sum_{k=1}^m\|\jump{\vc{u}_{h}^{m-1}}\|_{\vc{L}^2(\Omega)}^2 
		+ \sum_{k=1}^{m}\int_{\Omega} P''(\xi^{k-\frac{1}{2}}(\varrho_{h}^k, \varrho_{h}^{k-1}))\jump{\varrho^{k-1}_{h}}^2\ dx \\
		&\qquad	+ \sum_{k=1}^{m}\sum_{\Gamma \in \Gamma_h^I}\Delta t
		\int_{\Gamma}P''(\varrho^k_{\dagger})\jump{\varrho^k_{h}}_{\Gamma}^2
		\abs{\vc{u}^k_{h} \cdot \nu}\ dx.
	\end{split}
\end{equation*}
%Thus,  $\varrho_{h}\inb L^\infty(0,T;L^\gamma(\Omega))$.
\end{lemma}
\begin{proof}
The proof is almost identical to that of Lemma 5.3 in \cite{Karlsen1} 
and follows directly from standard arguments. We omit the details. 
\end{proof}

Since the finite element spaces are not conforming in $W^{1,2}(\Om)$ it is not clear that
the velocity and vorticity are embedded in $L^{2^*}(\Om)$. Knowing this is 
essential for the later convergence analysis. 
% Using the above Hodge decomposition together with Lemma \ref{lemma:embedding} we now prove 
% spatial embedding of the numerical solution in $L^{2^*}$. 

\begin{lemma}
\label{lemma:embeddingvelocity}
\solutiontext Then 
\begin{equation*}
\vc{w}_{h} \inb L^2(0,T;\vc{L}^{2^*}(\Omega)), \qquad \vc{u}_{h} \inb L^2(0,T; \vc{L}^{2^*}(\Omega)),
\end{equation*}
where $2^* = 6$ if $N=3$ and $2^*$ is any large finite number if $N=2$.
\end{lemma}
\begin{proof}
The second equation in  \eqref{FEM:momentumeq} with test function $\vc{\eta}_h = \Grad s_h$ reads:
\begin{equation*}
	 \int_\Om \vc{w}^m_h \Grad s_h~dx = 0, \quad \forall s_h \in S_h(\Om), \quad m=1, \ldots, M,
\end{equation*}
where the space $S_h(\Om)$ is defined in Section 2.3.
By definition, this means that $\vc{w}_h^m \in \vc{W}_h^{0,\perp}(\Om)$ and hence 
Lemmma \ref{lemma:embedding} is applicable and yields the desired estimate: 
\begin{equation}\label{eq:whbound}
	\vc{w}_{h} \inb L^2(0,T;\vc{L}^{2^*}(\Omega)).
\end{equation}

Next, we make use of Lemma \ref{lemma:hodge} and let $\{\vc{\zeta}_h\}_{h>0}$, $\{\vc{z}_h\}_{h>0}$
satisfy
\begin{equation*}
	\begin{split}
		& \vc{u}_{h}(\cdot,t) = \Curl \vc{\zeta}_{h}(\cdot, t) + \vc{z}_{h}(\cdot,t), \\
		& \vc{\zeta}_{h}(\cdot,t) \in \vc{W}_{h}^{0,\perp}(\Om), 
		\quad 
		\vc{z}_{h}(\cdot, t) \in \vc{V}_{h}^{0,\perp}(\Om),		
	\end{split}
\end{equation*}
for all $t \in (0,T)$.	

Another application of Lemma \ref{lemma:embedding}  yields
\begin{equation}\label{vel1}
	\vc{z}_{h} \inb L^2(0,T;\vc{L}^{2^*}(\Omega)).	
\end{equation}

Fix $\vc{\eta} \in \vc{W}_{0}^{\Curl,(2^*)'}(\Omega)$  and let $\vc{\eta}_{h} \in \vc{W}^{0,\perp}_{h}(\Omega)$ 
satisfy
\begin{equation*}
\int_{\Omega}\Curl \vc{\eta}_{h}\Curl \vc{\phi}_{h}\ dxdt =  \int_{\Omega}\Curl \vc{\eta} \Curl \vc{\phi}_{h} \ dxdt, \quad \forall \vc{\phi}_{h} \in \vc{W}_{h}(\Omega).
\end{equation*}
Then, by utilizing the second equation in \eqref{FEM:momentumeq} with $\vc{\eta}_{h}$ as test function
(the second equality below), we calculate
\begin{equation}\label{eq:pirk}
\begin{split}
\sum_{m=1}^M \Delta t \left|\int_{\Omega}\Curl \vc{\eta} \Curl \vc{\zeta}^m_{h} \ dx \right|^2 
& = \sum_{m=1}^M \Delta t\left|\int_{\Omega}\Curl \vc{\eta}_{h}\Curl \vc{\zeta}^m_{h}\ dx\right|^2 \\
 = \sum_{m=1}^M \Delta t  \left| \int_{\Omega}\vc{w}^m_{h}\vc{\eta}^m_{h} dxdt\right|^2 
& \leq \|\vc{w}_{h}\|^2_{L^2(0,T;\vc{L}^{2^*}(\Omega))}\|\vc{\eta}_{h}\|^2_{\vc{L}^{(2^*)'}(\Omega))} \\
& \leq \|\vc{w}_{h}\|^2_{L^2(0,T;\vc{L}^{2^*}(\Omega))}\|\Curl \vc{\eta}\|^2_{\vc{L}^{(2^*)'}(\Omega)},
\end{split}
\end{equation}
where the last inequality follows from the discrete Poincar\'e inequality \eqref{eq:Poincare2}. 

Now, for an arbitrary $\vc{\phi} \in L^{(2^*)'}$ let $\Curl \vc{\eta}$ be given through
the Hodge decomposition $\vc{\phi} = \Curl \vc{\eta} + \Grad \lambda$.
Then, we can use \eqref{eq:pirk} to deduce
\begin{equation*}
	\begin{split}
		&\int_{0}^T \left(\sup_{\vc{\vc{\phi}} \in L^{(2^*)'}(\Omega)}~
		\frac{\left|\int_{\Omega}\vc{\phi}\Curl \vc{\zeta}_{h} dx\right|}{\|\vc{\phi}\|_{\vc{L}^{(2^*)'}(\Omega)}}\right)^2~dt  \\
		&\qquad  = \int_{0}^T \left(\sup_{\vc{\vc{\phi}} \in L^{(2^*)'}(\Omega)}
		\frac{\left|\int_{\Omega}\Curl \vc{\eta}\Curl \vc{\zeta}_{h} dx\right|}{\|\vc{\phi}\|_{\vc{L}^{(2^*)'}(\Omega)}}\right)^2 ~dt \\
	    &\qquad  = \int_{0}^T \left(\sup_{\vc{\vc{\phi}} \in L^{(2^*)'}(\Omega)}
		\frac{\left|\int_{\Omega} \vc{\eta}_h \vc{w}_{h} dx\right|}{\|\vc{\phi}\|_{\vc{L}^{(2^*)'}(\Omega)}}\right)^2 ~dt 
				\leq \|\vc{w}_{h}\|^2_{L^2(0,T;\vc{L}^{2^*}(\Omega))},
	\end{split}
\end{equation*}
where the last term is bounded from \eqref{eq:whbound}.
Hence, $\Curl \vc{\zeta}_h \in_b L^2(0,T;\vc{L}^{2^*}(\Om))$ and, 
keeping in mind \eqref{vel1},  $\vu_h \in_b L^2(0,T;\vc{L}^{2^*}(\Om))$.
\end{proof}

In the upcoming convergence analysis and in order to establish weak time-continuity 
of the density we shall need to control the artificial diffusion introduced 
by the upwind discretization of the continuity equation. 
The following lemma provides the required bound.

\begin{lemma}
\label{lemma:productionbound}
\solutiontext Then
there exists a constant $C>0$ depending only on the initial energy $\mathcal{E}(\vrho^0, \vc{u}^{0})$,
the shape--regularity of $E_{h}$, $T$, and $|\Om|$, such that
\begin{equation*}
\begin{split}
&\sum_{E \in E_{h}}\int_{0}^T \int_{\partial E}\jump{\varrho_{h}}(\vc{u}_{h} \cdot \nu)^-(\Pi_{h}^Q\phi - \phi)\ dS(x)dt \\
&\leq h^{\theta(\gamma)}C\|\Grad \phi\|_{L^2(0,T;L^{2^*}(\Omega))}, \quad \forall \phi \in L^2(0,T;W^{1,2^*}(\Omega)),
\end{split}
\end{equation*}
where $2^* = 6$, if $N=3$, and $2^*$ is a sufficiently large number, if $N=2$. 
Here, $\theta(\gamma)>0$ is given by \eqref{eq:thetaeq} below.
\end{lemma}

\begin{proof}
Let $\phi \in L^2(0,T;W^{1,2^*}(\Omega))$ be arbitrary and set 
$$
	\phi^m = \frac{1}{\Delta t}\int_{t^{m-1}}^{t^m} \phi (s,x)~ds, \quad \phi_h^m = \Pi_h^Q \phi^m, \quad m=1, \ldots, M.
$$
We will need the auxilary function $B(z)= z^\alpha$.
where  
\begin{equation*}
\alpha = \frac{\gamma}{i+1} \ \textrm{and } i \in \mathbb{N} \textrm{ is chosen such that }\gamma \in (i+1,i+2].
\end{equation*}
Using $B''(z)>0$ and the H\"older inequality, we obtain
\begin{equation*}
	\begin{split}
		I^2 &:= \left|\sum_{m=1}^M\sum_{E \in E_{h}} \Delta t
					 \int_{\binner}\jump{\varrho^m_{h}}_{\partial E}(\vc{u}^m_{h} \cdot \nu)^-(\phi_h^m - \phi^m) dS(x)\right|^2 \\
		    &\leq \left(\sum_{m=1}^M\sum_{E \in E_{h}} \Delta t
					\int_{\binner }B''(\varrho^m_\dagger)\jump{\varrho^m_{h}}^2|\vc{u}^m_{h} \cdot \nu|dS(x)\right) \\
			&\qquad \times \left(\sum_{m=1}^M\sum_{E \in E_{h}} \Delta t\
					\int_{\binner}\left(B''(\varrho^m_\dagger)\right)^{-1}|\vc{u}^m_{h} \cdot \nu||\Pi_{h}^Q \phi^m - \phi^m|^2dS(x)\right) \\
			&=: I_1 \times I_2.
	\end{split}
\end{equation*}
In the case $\frac{N}{2} < \gamma \leq 2$, $\alpha = \gamma$ and  Lemma \ref{lemma:stability} yields 
\begin{equation*}
\label{eq:errorbound2}
I_1 \leq C\int_\Om B(\vr_0)~dx = C\int_\Om (\vr^0)^\gamma~dx.
\end{equation*}
Conversely, if $\gamma > 2$ then $2\alpha \leq \gamma$ 
and the renormalized scheme \eqref{FEM:renormalized} with $\phi_{h}:=1$ yields
\begin{equation*}%\label{eq:errorbound2half}
	\begin{split}
		I_1&\leq (\alpha - 1)\left| \sum_{m=1}^M \Delta t \int_{\Omega}
					(\varrho^m_{h})^\alpha \Div \vc{u}^m_{h} dx \right|Ê
				+ \int_{\Omega}(\vrho^0)^\alpha dx \\
		&\leq C\left(\|\varrho_{h}\|^\alpha_{L^\infty(0,T;L^\gamma(\Omega))}
					 \|\Div \vc{u}_{h}\|_{L^2(0,T;L^2(\Omega))} 
				+ \int_{\Omega}(\vrho^0)^\gamma dx\right),		
	\end{split}
\end{equation*}
which is bounded by Lemma \ref{lemma:stability}. Consequently, in both cases, we 
conclude that
\begin{equation}\label{inni}
	I_1 \leq C.
\end{equation}

To bound the $I_2$ term, we  utilize the  H\"older inequality:
\begin{equation}\label{eq:forkenneth}
	\begin{split}
		&\int_0^T\int_\Om |fgh^2|~dxdt \\
		&\qquad \leq \int_0^T\left(\int_\Om |fg|^{\frac{2^*}{2^* -2}}~dx \right)^\frac{2^* -2}{2^*}\left(\int_\Om |h|^{2^*}~dx \right)^\frac{2}{2^*}dt \\
		&\qquad \leq \int_0^T\left( \int_\Om |f|^{m_1}~dx\right)^\frac{1}{m_1}\|g\|_{L^{2^*}(\Om)}\|h\|_{L^{2^*}(\Om)}^2~dt\\
		&\qquad \leq \|f\|_{L^\infty(0,T;L^{m_1}(\Om))}\|g\|_{L^\infty(0,T;L^{2^*}(\Om))}\|h\|_{L^2(0,T;L^{2^*}(\Om))}^2,
	\end{split}
\end{equation}
where $1< m_1 =  \frac{2^*}{2^* -3}\leq 2$ and $\frac{1}{m_1} + \frac{1}{2^*} + \frac{2}{2^*} = 1$.

Now, using  \eqref{eq:forkenneth}, we deduce
\begin{equation}\label{eq:errorbound3}
	\begin{split}
I_2&\leq \alpha(\alpha-1)
			\max_{m=1, \ldots, M}\left(\sum_{E \in E_{h}} \int_{\binner}|\vc{u}^m_{h} \cdot \nu|^{2}~dS(x)\right)^\frac{1}{2} \\
	&\qquad \qquad  \times 
	\sum_{m=1}^M\Delta t\left(\sum_{E \in E_{h}}  \int_{\binner}\left|\Pi_{h}^Q \phi^m - \phi^m\right|^{2^*}~dS(x)\right)^\frac{2}{2^*} \\ 		 
	&\qquad \qquad \times 
	\max_{m=1, \ldots, M}\left(\sum_{E \in E_{h}}
		\int_{\binner}\left|\left(B''(\varrho^m_\dagger)\right)^{-1}\right|^{m_1}~dS(x)\right)^{\frac{1}{m_1}}.
\end{split}
\end{equation}
Next, we apply Lemma \ref{lemma:edgebounds} to deduce
\begin{equation}\label{eq:prod-i21}
	\begin{split}
		\max_{m=1, \ldots, M}\left(\sum_{E \in E_{h}} \int_{\binner}|\vc{u}^m_{h} \cdot \nu|^{2^*}~dS(x)\right)^\frac{1}{2^*}
		\leq
		Ch^{-\frac{1}{2^*}}\|\vu_h\|_{L^\infty(0,T;\vc{L}^{2^*}(\Om))}.
	\end{split}
\end{equation}
Similarly, we find that
\begin{equation*}%\label{eq:prod-i22}
	\begin{split}
		&\sum_{m=1}^M\Delta t\left(\sum_{E \in E_{h}}  \int_{\binner}\left|\Pi_{h}^Q \phi^m - \phi^m\right|^{2^*}~dS(x)\right)^\frac{2}{2^*} \\
		&\qquad \leq Ch^{-\frac{2}{2^*}}\|\Pi_h^Q\phi- \phi\|_{L^2(0,T;L^{2^*}(\Om))}^2 
		\leq Ch^{2 - \frac{2}{2^*}}\|\Grad\phi\|_{L^2(0,T;L^{2^*}(\Om))}^2,
	\end{split}
\end{equation*}
where the last inequality is an application of Lemma \ref{lemma:interpolation}.

To derive a similar bound for the $B''$ term in \eqref{eq:errorbound3}, 
we first note that,
since $\vr_h^m$ is everywhere positive and $2-\alpha <1$, 
\begin{equation*}
\left|\left(B''(\varrho^m_\dagger)\right)^{-1}\right|^{m_1} \leq \left|\vr_+^m + \vr_-^m \right|^{(2-\alpha)m_1}
\leq C(1+ \abs{\vr_+^m}^{m_1} + \abs{\vr_-^m}^{m_1}),
\end{equation*}
on every $\Gamma \cap \binner$.
From this, we conclude that
\begin{equation*}
	\begin{split}
		\int_{\partial E}\left|\left(B''(\varrho^m_\dagger)\right)^{-1}\right|^{m_1}~dS(x)
		&\leq h^{-1}C\left(|E| + \int_{\mathcal{N}(E)\cup E}|\vr_h|^{m_1}~dx \right),
	\end{split}
\end{equation*}
where $\mathcal{N}(E)$ denotes the union of the neighboring elements of $E$.
Applying this together with Lemma \ref{lemma:edgebounds}, we obtain
\begin{equation}\label{eq:prod-i23}
	\begin{split}
		&\max_{m=1, \ldots, M}\left(\sum_{E \in E_{h}}
			\int_{\binner}\left|\left(B''(\varrho^m_\dagger)\right)^{-1}\right|^{m_1}~dS(x)\right)^{m_1}	 \\
		&\qquad \leq Ch^{-\frac{1}{m_1}}\left(|\Om|^\frac{1}{m_1} +\|\vr_h\|_{L^\infty(0,T;L^{m_1}(\Om))}\right) \\
		&\qquad \leq Ch^{-\frac{1}{m_1}}\left( 1 + h^{\min\{0,N(\frac{1}{m_1} - \frac{1}{\gamma})\}}\|\vr_h\|_{L^\infty(0,T;L^\gamma(\Om))}\right),
	\end{split}
\end{equation}
where the last inequality is a standard inverse estimate (Lemma \ref{lemma:inverse}).
Setting \eqref{eq:prod-i21}--\eqref{eq:prod-i23} into \eqref{eq:errorbound3} leads to the bound
\begin{equation}\label{eq:errorbound4}
	\begin{split}
	 	I_2 &\leq Ch\|\vu_h\|_{L^\infty(0,T;\vc{L}^{2^*}(\Om))}\|\Grad\phi\|^2_{L^2(0,T;L^{2^*}(\Om))}\\
			&\qquad \qquad \times		\left( 1 + h^{\min\{0,N(\frac{1}{m_1} - \frac{1}{\gamma}\}}\|\vr_h\|_{L^\infty(0,T;L^\gamma(\Om))}\right) \\
			&\leq Ch^\frac{1}{2}\|\vu_h\|_{L^2(0,T;\vc{L}^{2^*}(\Om))}		\|\Grad\phi\|^2_{L^2(0,T;L^{2^*}(\Om))} \\
			&\qquad \qquad \times		\left( 1 + h^{\min\{0,N(\frac{1}{m_1} - \frac{1}{\gamma}\}}\|\vr_h\|_{L^\infty(0,T;L^\gamma(\Om))}\right),
	\end{split}
\end{equation}
where the last inequality is an application Lemma \ref{lemma:inverse} in time (keeping in mind $\Delta t = \kappa h$).
We have also used that
$$
h^{-(\frac{1}{m_1} + \frac{1}{2^*} + \frac{2}{2^*})} = h^{-1}.
$$  

In 2D, $2^*$ is any large finite number. Consequently, we can always make sure that $m_1 \leq \gamma$.
Using this, it is straight forward to check that
\begin{equation}\label{eq:thereq}
	h^\frac{1}{2}h^{\min\{0,N(\frac{1}{m_1} - \frac{1}{\gamma}\}} = h^{2\theta(\gamma)},
\end{equation}
where %$\theta(\gamma)$ is given by
\begin{equation}\label{eq:thetaeq}
	\begin{split}
		0 < \theta(\gamma) := \begin{cases}
						\frac{1}{4}, & N=2, \\
						\frac{1}{2} + \min\{0,3(\frac{1}{2} - \frac{1}{\gamma})\},& N=3,
					\end{cases}	
	\end{split}
\end{equation}
By setting \eqref{eq:thereq} into \eqref{eq:errorbound4} and applying Lemma \ref{lemma:stability}, 
we obtain
\begin{equation*}
	I_2 \leq h^{2\theta(\gamma)}C\|\Grad\phi\|^2_{L^2(0,T;L^{2^*}(\Om))}.
\end{equation*}
This and \eqref{inni} gives
\begin{equation*}
	I^2 = I_1 \times I_2 \leq C\|\Grad\phi\|_{L^2(0,T;L^{2^*}(\Om))}h^{2\theta(\gamma)},
\end{equation*}
which brings the proof to an end.
\end{proof}

\subsection{Weak time-continuity estimates}
We end this section by establishing weak time-continuity of 
the density and velocity.
\begin{lemma}\label{lemma:timecont}
\solutiontext
Then
$$
\Pth{\vr_h}\inb L^2(0,T;W^{-1,(2^*)'}(\Omega)),
$$
where $(2^*)' = \frac{2^*}{2^*-1}$ and $2^*$ is 
as in the previous lemma. 
\end{lemma}

\begin{proof}
The proof is almost identical to the proof of 
Lemma 5.6 in \cite{Karlsen1} and is only 
included for the sake of completeness. 

Fix $\phi \in L^2(0,T;W^{1,2^*}(\Omega))$, and 
introduce the piecewise constant approximations 
$\phi_{h}:=\Pi_{h}^Q \phi$, $\phi_{h}^m:= \Pi_{h}^Q \phi^m$, and
$\phi^m:=\frac{1}{\Delta t}\int_{t^{m-1}}^{t^m} \phi(t,\cdot)\ dt$.

The continuity scheme \eqref{FEM:contequation} 
with $\phi_{h}^m$ as test function reads
\begin{equation}\label{eq:tcont1}
	\begin{split}
		& \Delta t\int_{\Omega} \Pth{\vr_h^m}\phi^m\ dxdt 
		\\ & \qquad 
		= \Delta t\sum_{\Gamma \in \Gamma^I_{h}}
		\int_\Gamma\left(\vrho^m_-(\vc{u}^{m}_h \cdot \nu)^+ 
		+\vrho^m_+(\vc{u}^{m}_h \cdot \nu)^-\right)
		\jump{\phi^m_{h}}_\Gamma\ dS(x).
	\end{split}
\end{equation}
Since the traces of $\phi^m$ taken from either side of a 
face are equal, we can write
\begin{equation}\label{eq:densitycalc}
\begin{split}
	& \sum_{\Gamma \in \Gamma^I_{h}}
	\int_\Gamma \left(\vrho^m_-(\vc{u}^{m}_h \cdot \nu)^+ 
	+\vrho^m_+(\vc{u}^{m}_h \cdot \nu)^-\right)\jump{\phi^m_{h}}_\Gamma\ dx \\
	&=\sum_{\Gamma \in \Gamma^I_{h}}\int_{\Gamma} \left( \vrho^m_{+}(\vc{u}^m_{h} \cdot \nu)^- 
	+ \vrho^m_{-}(\vc{u}^m_{h} \cdot \nu)^+\right)\jump{\phi^m_{h} - \phi^m}\ dS(x), \\
	&= -\sum_{E \in E_{h}}\int_{\binner} \left( \vrho^m_{+}(\vc{u}^m_{h} \cdot \nu)^+ 
	+ \vrho^m_{-}(\vc{u}^m_{h} \cdot \nu)^-\right)(\phi^m_{h}- \phi^m)\ dS(x), \\
	&= \sum_{E \in E_h}\int_{E}- \Div (\vrho^m_{h}\vc{u}^m_{h}(\phi^m_{h} - \phi^m))\ dx \\
	&\qquad \qquad
	+ \sum_{E \in E_{h}}\int_{\binner }\jump{\vrho^m_{h}}_{\partial E}(\vc{u}^m_{h} \cdot \nu)^-
	(\phi^m_{h} - \phi^m)\ dS(x) \\
	&= \int_{\Omega}\vrho^m_{h}\vc{u}^m_{h}\cdot \Grad \phi^m\ dx  
	+ \sum_{E \in E_{h}}\int_{\binner}
	\jump{\vrho^m_{h}}_{\partial E}(\vc{u}^m_{h} \cdot \nu)^-
	(\phi^m_{h} - \phi^m)\ dS(x).
\end{split}
\end{equation}
To conclude the last equality, we have used
\begin{equation*}
\int_E \vrho_h^m \Div \vc{u}_h^m (\phi_h^m - \phi^m)~dx = (\vrho_h^m \Div \vc{u}_h^m)|_E \int_E \Pi_h^Q\phi^m -  \phi^m~dx = 0, \quad \forall E \in E_h,
\end{equation*}
since both $\vrho_h^m$ and $\Div \vc{u}_h^m$ are piecewise constant.

By summing \eqref{eq:tcont1} over $m$, taking 
absolute values, and using the above identity, we find
\begin{equation*}
	\begin{split}
		&\left|\sum_{m=1}^M\Delta t \int_{\Omega} 
		\Pth{\vr_h^m}\phi^m\ dxdt\right| \\
		&\qquad \qquad \leq \left|\sum_{m=1}^M \Delta t \int_{\Omega}\vrho_{h}^m 
		\vc{u}_{h}^m \Grad \phi^m\ dx\right| \\
		&\qquad \qquad \qquad+\left|\sum_{m=1}^M \sum_{E \in E_{h}}\Delta t
		\int_{\binner}\jump{\vrho^m_{h}}_{\partial E}(\vc{u}^m_{h} \cdot \nu)^-
		(\phi_{h}^m - \phi^m)\ dS(x)\right|.
	\end{split}
\end{equation*}
Using Lemma \ref{lemma:productionbound}, together with an 
application of H\"older's inequality, we deduce
\begin{align*}
	&\left|\sum_{m=1}^M\Delta t \int_{\Omega} 
	\Pth{\vr_h^m}\phi^m\ dx\right| \\
	&\leq \sum_{m=1}^M \Delta t \|\vr_h^m\|_{L^\alpha(\Om)}\|\vu_h^m\|_{\vc{L}^{2^*}(\Om)}\|\Grad \phi^m\|_{\vc{L}^{2^*}(\Om)}  
			+Ch^{\theta(\gamma)}\|\Grad \phi\|_{L^2(0,T;\vc{L}^{2^*}(\Om))} \\
	&\leq \|\vr_h\|_{L^\infty(0,T;L^\alpha(\Om))}\|\vu_h\|_{L^\infty(0,T;\vc{L}^{2^*}(\Om))}\|\Grad \phi\|_{L^2(0,T;\vc{L}^{2^*}(\Om))}  \\
	&\qquad \qquad + C \, h^{\theta(\gamma)}
	\norm{\Grad \phi}_{L^2(0,T;\vc{L}^{2^*}(\Omega))}, 
\end{align*}
where $\alpha = \frac{2^* (2^*)'}{2^* - (2^*)'} < \gamma$ since $\gamma > \frac{N}{2}$ 
and $\frac{1}{\alpha} + \frac{1}{2^*} + \frac{1}{2^*} = 1$.
By Lemma \ref{lemma:stability}, the right--hand side is bounded, so we conclude that
\begin{align*}
	&\left|\int_{\Delta t}^T\int_{\Omega}
		\Pth{\vr_h}\phi\ dxdt\right| \\
	&=\left|\sum_{m=1}^M\Delta t \int_{\Omega} 
	\Pth{\vr_h^m}\phi^m\ dx\right| 
	\leq C\, (1+h^{\theta(\gamma)})
	\norm{\Grad \phi}_{L^2(0,T;\vc{L}^{2^*}(\Omega))}.
\end{align*}
This brings the proof to an end.
\end{proof}

\begin{lemma}\label{lemma:velocitycont}
\solutiontext Then
\begin{equation*}
\Pth{\vc{u}_{h}} \inb L^{2}(0,T;W^{-1,1}(\Omega)),
\end{equation*}

\end{lemma}
\begin{proof}
By adding and subtracting, we see that for any $\vc{\phi} \in L^2(0,T;\vc{W}^{1,\infty}_{0}(\Omega))$,
\begin{align*}
&\int_{0}^T\int_{\Omega}\Pth{\vc{u}_{h}}\vc{\phi}\ dxdt \\
&= \int_{0}^T\int_{\Omega}\Pth{\vc{u}_{h}} \Pi_{h}^V\vc{\phi}\ dxdt
+ \int_{0}^T\int_{\Omega}\Pth{\vc{u}_{h}} \left(\vc{\phi}- \Pi_{h}^V \vc{\phi}\right)\ dxdt.
\end{align*}
From the first equation of the momentum scheme \eqref{FEM:momentumeq} with $\vc{v}_{h} = \Pi_{h}^V \vc{\phi}$, we have that
\begin{equation*}
\begin{split}
\int_{0}^T\int_{\Omega}\Pth{\vc{u}_{h}} \Pi_{h}^V\vc{\phi} dxdt
&= -\int_{0}^T \int_{\Omega}\mu \Curl \vc{w}_{h}\left(\Pi_{h}^V \vc{\phi}\right) + (\mu + \lambda)\Div \vc{u}_{h} \Div \Pi_{h}^V \vc{\phi}\ dxdt \\
&\qquad +\int_{0}^T\int_{\Omega} a\varrho_{h}^\gamma \Div \Pi_{h}^V \vc{\phi} \ dxdt\\
&\leq C\left(\|\Curl \vc{w}_{h}\|_{L^2(0,T;\vc{L}^2(\Omega)}\|\vc{\phi}\|_{L^2(0,T;\vc{L}^2(\Omega)} \right.\\
&\qquad + \|\Div \vc{u}_{h}\|_{L^2(0,T;L^2(\Omega))}\|\Div \vc{\phi}\|_{L^2(0,T;L^2(\Omega))}  \\
& \qquad +\left. \|\varrho_{h}\|_{L^{\infty}(0,T;L^{\gamma }(\Omega))}\|\Div \vc{\phi}\|_{L^{1}(0,T;L^{\infty}(\Omega))}\right) \\
&\leq C \| \vc{\phi}\|_{L^2(0,T;\vc{W}^{1,\infty}(\Om))},
\end{split}
\end{equation*}
where the last inequality follows from Lemma \ref{lemma:stability}.

From Lemma \ref{lemma:stability},  we also have the estimate
\begin{equation}\label{eq:timediv}
\|\Pth{\vc{u}_{h}}\|_{L^2(0,T;\vc{L}^2(\Omega))} 
= (\Delta t)^{-\frac{1}{2}}\left(\sum_{m=1}^M \int_{\Om}\jump{\vc{u}_{h}^{m-1}}^2\ dx\right)^\frac{1}{2} 
\leq h^{-\frac{1}{2}}C.
\end{equation}
Using \eqref{eq:timediv}, we estimate
\begin{equation*}
\begin{split}
\int_{\Delta t}^T \int_{\Omega}\frac{d\left(\Pi_{\mathcal{L}}\vc{u}_{h}\right)}{dt}\left(\vc{\phi}- \Pi_{h}^V \vc{\phi}\right) dxdt
&\leq C\|\Pth{\vc{u}_{h}}\|_{L^2(0,T;\vc{L}^2(\Omega))} \|\vc{\phi}- \Pi_{h}^V \vc{\phi}\|_{L^2(0,T;\vc{L}^2(\Omega)} \\
& \leq C\frac{h}{\sqrt{\Delta t}}\|\Grad\phi\|_{L^2(0,T;\vc{L}^2(\Omega))} 
\leq Ch^\frac{1}{2}, 
\end{split}
\end{equation*}
where we in the last inequality  have used the relation $\Delta t = \kappa h$. 
Combining the previous estimates concludes the proof.
\end{proof}

Recall our notation for the Hodge decomposition of the solution $\vu$,
$$
\vu = \Curl \vc{\zeta} + \Grad s.
$$
In the next result, we prove that $\partial_t (\Curl \vc{\zeta}) \in L^2(0,T;\vc{L}^2(\Om))$.
To see why such a bound is reasonable, apply the $\Curl$ operator to 
the velocity equation \eqref{eq:vorticity-form} 
\begin{equation*}
	\Curl (\Curl \vc{\zeta})_t + \mu\Curl \Curl \vc{w} = 0.
\end{equation*}
Multiplying with $\vc{\zeta}_t$, integrating by parts in space,
and applying  H\"older's inequality, 
\begin{equation*}
	\|\Curl \vc{\zeta}_t\|_{L^2(\Om)}^2 
		\leq \epsilon \|\Curl \vc{\zeta}_t\|_{L^2(\Om)}^2
		 	+ \frac{C}{\epsilon}\|\Curl \vc{w}\|_{\vc{L}^2(\Om)}^2.
\end{equation*}
Fixing $\epsilon$ small, and integrating in time
\begin{equation*}
	\int_0^T \|\Curl \vc{\zeta}_t\|_{L^2(\Om)}^2~dt \leq C\int_0^T\|\Curl \vc{w}\|_{\vc{L}^2(\Om)}^2~dt,
\end{equation*}
where the right-hand side is bounded. Consequently, it is the higher regularity on 
$\vc{w} = \Curl \vu$ that enable us to obtain the bound.

\begin{lemma}\label{lem:lpbound}
\solutiontext 
Let $\Set{(\vc{\zeta}_{h},\vc{z}_{h})}_{h>0}$ be the sequence given by the decomposition
$\vc{u}_{h}(\cdot,t) = \Curl \vc{\zeta}_{h}(\cdot, t) + \vc{z}_{h}(\cdot,t)$
and $\vc{\zeta}_{h}(\cdot,t) \in \vc{W}_{h}^{0,\perp}(\Om)$, $\vc{z}_{h}(\cdot, t) \in \vc{V}_{h}^{0,\perp}(\Om)$, for $t \in (0,T)$.
Then 
\begin{equation*}
\Pth{\Curl \vc{\zeta}_{h}} \inb L^2(0,T;\vc{L}^2(\Omega)).
\end{equation*}
\end{lemma}
\begin{proof}
For any $m=1, \ldots, M$, let $\vc{v}_{h}^m = \Pth{\Curl \vc{\zeta}_{h}^m} \in \vc{V}_h$. 
Observe that by the orthogonality of the Hodge decomposition,
\begin{equation*}
	\int_\Om \Pth{\vu_h^m} \Pth{\Curl \vc{\zeta}_{h}^m}~dx =  \int_{\Omega}\abs{\Pth{\Curl \vc{\zeta}_{h}^m}}^2\ dx.
\end{equation*}
Hence, by setting $\vc{v}_{h}^m$ as test function in the first equation of the momentum scheme
\eqref{FEM:momentumeq}, multiplying with $\Delta t$, and summing over all $m=1, \ldots, M$, 
we obtain
\begin{equation*}
\begin{split}
&\sum_{m=1}^M \Delta t \int_{\Omega}\abs{\Pth{\Curl \vc{\zeta}_{h}^m}}^2\ dxdt \\
&\qquad = -\sum_{m=1}^M\Delta t \int_{\Omega}\mu \Curl \vc{w}_{h}^m\Pth{\Curl \vc{\zeta}_{h}^m}\ dxdt \\
&\qquad \leq \mu\left(\sum_{m=1}^M \Delta t \|\Pth{\Curl \vc{\zeta}_{h}^m}\|^2\right)^\frac{1}{2}\left(\sum_{m=1}^M\Delta t\|\Curl \vc{w}^m_{h}\|_{\vc{L}^2(\Omega)}^2\right)^\frac{1}{2}.
\end{split}
\end{equation*}

An application of the Cauchy inequality with $\epsilon$ to the above estimate
yields
\begin{equation*}
\|\Pth{\Curl \vc{\zeta}_{h}}\|_{L^2(0,T;\vc{L}^2(\Omega))} \leq \frac{\mu}{2}\|\Curl \vc{w}_{h}\|_{L^2(0,T;\vc{L}^2(\Omega))}.
\end{equation*}
Lemma \ref{lemma:stability} provides a bound on the right--hand side and 
hence the proof is complete.

\end{proof}

\section{Higher intergrability on the density}\label{sec:higherint}
The stability estimate only provides the 
bound $p(\vrho_{h}) \inb L^\infty(0,T;L^1(\Omega))$.
Hence, it is not clear that $p(\vrho_{h})$ converges 
weakly to an integrable function. Moreover, 
the subsequent analysis relies heavily on the pressure 
having higher integrability. In this section
 we establish that the density is in fact bounded 
in $L^{\gamma+1}(0,T;L^{\gamma+1}(\Omega))$, independently of $h$.
The main technical tool used to achieve this is an equation for the effective viscous flux:
$$
	\eff(\varrho, \vc{u}) = p(\varrho) - (\lambda + \mu)\Div \vc{u}.
$$
We start by deriving this equation.
For this purpose,
fix any $\phi \in L^2(0,T;L_0^2(\Omega))$ 
and, for each fixed $h>0$, let
$$
	\vc{v}_{h}(t, \cdot) =  \Aoph{\phi}(t, \cdot),\quad t \in (0,T),
$$
and
$$
\vc{v}_{h}^m = \frac{1}{\Delta t}\int_{t^{m-1}}^{t^m}\vc{v}_{h}( s, \cdot)\ ds, \quad m =1, \ldots, M.
$$
Observe that $\vc{v}_h$ is constructed such that
$$
\Div \vc{v}_h^m = \frac{1}{\Delta t}\int_{t^{m-1}}^{t^m} \phi~dt, \quad m=1, \ldots, M.
$$

By inserting $\vc{v}_{h}^m$ as test function in the momentum scheme \eqref{FEM:momentumeq},
multiplying with $\Delta t$, and summing over
all $m=1, \ldots, M$, we are led to the identity
\begin{equation*}
\begin{split}
\int_{0}^{T}\int_{\Omega} \eff(\varrho_{h}, \vc{u}_{h}) \phi\ dxdt 
=\int_{0}^{T}\int_{\Om}\left(\Pth{\vc{u}_{h}} + \mu \Curl \vc{w}_{h}\right)  \Aoph{\phi}\ dxdt.
\end{split}
\end{equation*}
Since $\int_{\Omega} (\Curl \vc{w}_{h}^m) \Aop{\phi} dx = 0$, for all $m=1, \ldots, M$, we can further write
\begin{equation}
\label{eq:effvisc}
\begin{split}
&\int_{0}^{T}\int_{\Omega} \eff(\varrho_{h}, \vc{u}_{h}) \phi\ dxdt \\
&=\int_{0}^{T}\int_{\Om}\Pth{\vc{u}_{h}}   \Aop{\phi}\ dxdt \\
&\qquad \quad + \int_{0}^T \int_{\Omega}\left(\Pth{\vc{u}_{h}} + \mu \Curl \vc{w}_{h} \right)\left(\Aoph{\phi} - \Aop{\phi}\right)\ dxdt.
\end{split}
\end{equation}
As $\phi$  was fixed arbitrary, we can conclude that \eqref{eq:effvisc} holds for all
$\phi \in L^2(0,T;L_0^2(\Omega))$.
% The equation \eqref{eq:effvisc}
% will serve as a starting point for proving both higher integrability on the density and weak sequential
% continuity of the effective viscous flux.  

The following lemma ensures that the last term of \eqref{eq:effvisc} converges to zero.

\begin{lemma}\label{lemma:interpolationerrorbound} \solutiontext
Then there exists a constant $C>0$, depending only on the initial data
and the shape regularity of $E_h$, such that
\begin{equation*}
\begin{split}
&\left|\int_{0}^T \int_{\Omega}\left(\Pth{\vc{u}_{h}} + \mu \Curl \vc{w}_{h}\right)\left(\Aoph{\phi} - \Aop{\phi}\right)\ dxdt\right| \\
& \leq C(h^\frac{1}{2} + h)\|\phi\|_{L^2(0,T;L^2(\Omega))}, \quad \forall \phi \in L^2(0,T;L^2_{0}(\Omega)).
\end{split}
\end{equation*}
%$E_{0}= E(\vrho^0, \vc{u}^{0})$ is given by \eqref{eq:energydef}.
\end{lemma}
\begin{proof}
%where we have also utilized the assumption $\Delta t = \kappa h$.

By this, the H\"older inequality, and  Lemma \ref{lemma:interpolation}, 
we deduce
\begin{equation*}
\begin{split}
&\left|\int_{0}^{T}\int_{\Om}\left(\Pth{\vc{u}_{h}} + \mu \Curl \vc{w}_{h}\right)  (\Aoph{\phi} - \Aop{\phi})\ dxdt\right| \\
&\qquad \leq ch\|\Grad\Aop{\phi}\|_{L^2(0,T;\vc{L}^2(\Omega))} \\
&\qquad \qquad \times\left(\|\Pth{\vc{u}_{h}}\|_{L^2(0,T;\vc{L}^2(\Omega))} + \|\Curl \vc{w}_{h}\|_{L^2(0,T;\vc{L}^2(\Omega))} \right)\\
&\leq C(h^\frac{1}{2} + h)\|\phi\|_{L^2(0,T;L^2(\Omega))},
\end{split}
\end{equation*}
where we in the last inequality have used Lemma \ref{lemma:stability} and \eqref{eq:timediv}.
\end{proof}

We are now in a position to prove higher integrability of the density. 
To increase readability of the proof, we introduce the notation
$$
\avg{\phi} = \frac{1}{|\Omega|}\int_{\Omega}\phi \ dx,
$$
for the spatial average value of a function.

\begin{lemma}[Higher integrability on the density] \label{lemma:higherorderpressure} \solutiontext
Then
\begin{equation*}
\varrho_{h} \inb L^{\gamma + 1}((0,T) \times \Omega).
\end{equation*}
\end{lemma}

\begin{proof}
Setting $\phi = \varrho_{h} - \avg{\vr^0_h}$ in \eqref{eq:effvisc} yields the identity
\begin{equation*}
%\label{eq:forgotten}
\begin{split}
&\int_{0}^{T}\int_{\Omega} p(\varrho_{h})\varrho_{h}\ dxdt \\
& \qquad = \int_{0}^T \int_{\Omega} p(\varrho_{h})\avg{\varrho_{h}^0}  +   (\lambda + \mu) \Div \vc{u}_{h}\varrho_{h} 
 +\Pth{\vc{u}_{h}}\Aop{(\varrho_{h} - \avg{\varrho^{0}_h})}\ dxdt \\
&\qquad \qquad + \int_{0}^T \int_{\Omega}\left(\Pth{\vc{u}_{h}} + \mu \Curl \vc{w}_{h}\right) \\
& \qquad \qquad  \qquad \qquad
\times \left(\Aoph{\varrho_{h} - \avg{\varrho_{h}^0}} - \Aop{\varrho_{h} - \avg{\varrho_{h}^0}}\right)\ dxdt,
\end{split}
\end{equation*}
%where we have also used the fact that $\psi(t) \Aoph{\phi}= \Aoph{\psi(t)\phi}$ and $\psi(t) \Aop{\phi}= \Aop{\psi(t)\phi}$.
Applying the H\"older inequality and  Lemmas \ref{lemma:stability} and \ref{lemma:interpolationerrorbound} yields
\begin{equation}
\label{eq:pressure1}
\begin{split}
\left|\int_{0}^{T}\int_{\Omega} p(\varrho_{h})\varrho_{h}\ dxdt \right| 
&\leq \left| \int_{0}^T \int_{\Omega}\Pth{\vc{u}_{h}} \Aop{\varrho_{h} - \avg{\varrho_{h}^0}}\ dxdt \right| \\
& \qquad + C\left(1+ h^\frac{1}{2} + h\right)\|\varrho_{h}\|_{L^2(0,T;L^2(\Omega))}.
\end{split}
\end{equation}
To bound the first term on the right--hand side, we 
first note that
\begin{equation}
\label{eq:pressure2}
\begin{split}
&\int_{0}^T \int_{\Omega}\Pth{\vc{u}_{h}}  \Aop{\varrho_{h} - \avg{\varrho_{h}^0}}~dxdt \\
&\qquad \qquad
= \sum_{m=1}^M\Delta t \int_{\Omega}\frac{\vc{u}_{h}^m - \vc{u}_{h}^{m-1}}{\Delta t}\Aop{\varrho^m_{h} - \avg{\varrho_{h}^0}}~ dx.
\end{split}
\end{equation} 
%Keeping in mind that $\gamma' \leq 2^*$ since $\gamma > \frac{N}{2}$, 
Then, we apply summation by parts to \eqref{eq:pressure2} and make use of the H\"older inequality to obtain
\begin{equation}
\label{eq:pressure3}
\begin{split}
&\left|\int_{0}^T \int_{\Omega}\Pth{\vc{u}_{h}}  \Aop{\varrho_{h} - \avg{\varrho_{h}^0}}~dxdt\right| \\
&\qquad \qquad = \left|-\sum_{k=1}^{M} \Delta t \int_{\Omega}\vc{u}_{h}^{m-1}\Aop{\Pth{\varrho_{h}^{m}}}\ dx\right. \\
&\qquad \qquad \qquad \qquad \left. 
	-  \frac{1}{\Delta t}\int_{0}^{\Delta t} \int_{\Omega} \vc{u}^{0}_h\Aop{\varrho_{h} - \avg{\varrho_{h}^0}} dxdt\right| \\
&\qquad \qquad \leq   \left|\sum_{m=1}^M \Delta t \int_{\Omega}\vc{u}_{h}^{m-1}\Aop{\Pth{\varrho_{h}^{m}}}\ dx\right| \\
&\qquad \qquad \qquad \qquad + C\|\vc{u}^{0}\|_{\vc{L}^2(\Omega)}\|\varrho_{h} - \avg{\varrho_{h}^0}\|_{L^\infty(0,T;L^\gamma(\Omega))},
\end{split}
\end{equation}
where we in the last inequality have used elliptic theory (and $\gamma^* > 2$ since $\gamma > \frac{N}{2}$) to conclude
that
$$
\norm{\Aop{\varrho_{h} - \avg{\varrho_{h}^0}}}_{L^\infty(0,T;\vc{L}^2(\Om))} \leq C\|\varrho_{h} - \avg{\varrho_{h}^0}\|_{L^\infty(0,T;L^\gamma(\Om))}.
$$

Next, using integration by parts, 
\begin{equation}
\label{eq:pressure4}
\begin{split}
&\sum_{m=1}^M \Delta t \int_{\Omega}\vc{u}_{h}^{m-1}\Aop{\Pth{\varrho_{h}^{m}}}\ dx \\
&= \sum_{m=1}^M \Delta t \int_{\Om}  \Delta^{-1}\left[\Div \vc{u}_{h}^{m-1}\right]\left(\frac{\varrho_{h}^m - \varrho_{h}^{m-1}}{\Delta t}\right) ~dx \\
&= \sum_{m=1}^M \Delta t \int_{\Om} \vc{u}_{h}^m \varrho_{h}^m \Aop{\Div \vc{u}^{m-1}_{h}} ~dx \\
&\qquad \quad + \sum_{m=1}^M \Delta t \sum_{E \in E_{h}}\int_{\partial E}
				\jump{\varrho_{h}^m}_{\partial E}(\vc{u}^m_{h} \cdot \nu)^-(\Pi_{h}^V - \mathbb{I})\Delta^{-1}\left[\Div \vc{u}^{m-1}_{h}\right]~dS(x),
\end{split}
\end{equation}
where the last equality is deduced as follows: Set $\phi^m_{h} = \Pi_{h}^Q \Delta^{-1}\left[\Div \vc{u}_{h}^{m-1}\right]$ in the continuity scheme \eqref{FEM:contequation}  
and  perform the calculation \eqref{eq:densitycalc}.

By setting \eqref{eq:pressure4} into \eqref{eq:pressure3},
% using that
% $\norm{\Grad\Delta^{-1}\left[\Div \vc{u}^{m-1}_{h}\right]}_{L^p(\Om)}\leq C\|\vu_h^{m-1}\|_{L^p(\Om)}$,
and applying Lemma \ref{lemma:productionbound}, we obtain
\begin{equation}\label{eq:this}
\begin{split}
&\left|\int_{0}^T \int_{\Omega}\Pth{\vc{u}_{h}}  \Aop{\varrho_{h} - \avg{\varrho_{h}^0}}\ dxdt\right| \\
&\qquad\leq  C\left(1+\|\vc{u}_{h}\|_{L^{2}(0,T;\vc{L}^{\frac{2\gamma}{\gamma-1}}(\Omega))}\|\varrho_{h}\|_{L^{\infty}(0,T;L^{\gamma}(\Omega))}\right)\\
&\qquad\qquad +h^{\theta(\gamma)}C\|\Aop{\Div \vc{u}_{h}}\|_{L^2(0,T;\vc{L}^{2^*}(\Omega))},
\end{split}
\end{equation}
where $\theta(\gamma)$ is given by \eqref{eq:thetaeq}.

Finally, inserting \eqref{eq:this} into \eqref{eq:pressure1} and recalling that $\frac{2\gamma}{\gamma - 1} < 2^*$, since $\gamma > \frac{N}{2}$,
gives
\begin{equation*}
\begin{split}
&\left|\int_{0}^{T}\int_{\Omega} a\varrho_{h}^{\gamma+1}\ dxdt \right| \\
&\leq C\left(1+\|\vc{u}_{h}\|_{L^{2}(0,T;\vc{L}^{2^*}(\Omega))}\|\varrho_{h}\|_{L^{\infty}(0,T;L^{\gamma}(\Omega))}\right)
 +h^{\theta(\gamma)}C\|\Div \vc{u}_{h}\|_{L^2(0,T;L^{2}(\Omega))} \\
&\qquad + \left(1+ h^\frac{1}{2} + h\right)\|\varrho_{h}\|_{L^2(0,T;L^2(\Omega))}.
\end{split}
\end{equation*}
The proof then follows from  the H\"older and  Cauchy (with epsilon) inequalities.
\end{proof}

\section{Convergence}\label{sec:conv}
\solutiontext
In this section we establish that a subsequence 
of $\{\left(\vrho_{h}, \vc{w}_{h}, \vc{u}_{h}\right)\}_{h>0}$ 
converges to a weak solution of the semi--stationary Stokes system, thereby proving 
Theorem \ref{theorem:mainconvergence}. The proof is divided into 
several steps: 
\begin{enumerate}
	\item{}Strong convergence of the velocity.
	\item{}Convergence of the continuity scheme.
	\item{}Weak sequential continuity of the discrete viscous flux.
	\item{}Strong convergence of the density.
	\item{}Convergence of the velocity scheme.
\end{enumerate}

Our starting point is that the results of the  previous sections ensure us that the approximate 
solutions $(\vc{w}_{h}, \vc{u}_{h}, \vrho_{h})$ satisfy the following 
$h$--independent bounds:
$$
\vrho_{h} \inb L^\infty(0,T;L^\gamma(\Omega))
\cap L^{\gamma+1}((0,T)\times \Omega),
$$
$$
\vc{w}_{h} \inb L^\infty(0,T;\vc{L}^2(\Om)) \cap L^2(0,T;\Hcurl), 
$$
$$
\vc{u}_{h} \inb L^\infty(0,T;\vc{L}^2(\Om))\cap L^2(0,T;\Hdiv).
$$
Moreover, in view of Lemma \ref{lemma:hodge}, there exists sequences $\{\vc{\zeta}_h\}_{h>0}$, $\{\vc{z}_h\}_{h>0}$
such that
\begin{equation}\label{eq:basic-decomp}
	\begin{split}
		& \vc{u}_{h}(\cdot,t) = \Curl \vc{\zeta}_{h}(\cdot, t) + \vc{z}_{h}(\cdot,t), \\
		& \vc{\zeta}_{h}(\cdot,t) \in \vc{W}_{h}^{0,\perp}(\Om), 
		\quad 
		\vc{z}_{h}(\cdot, t) \in \vc{V}_{h}^{0,\perp}(\Om),		
	\end{split}
\end{equation}
for all $t \in (0,T)$ where
$$
\vc{z}_h \inb L^\infty(0,T;\vc{L}^2(\Om))\cap L^2(0,T;\Hdiv).
$$
$$
\Curl \vc{\zeta}_h \inb L^\infty(0,T; \vc{L}^2(\Om)),
$$
and
$$
\Pth{\Curl \vc{\zeta}_h} \inb L^2(0,T; \vc{L}^2(\Om)).
$$

Consequently, we may assume that there exist functions 
$\vrho,\vc{w},\vc{u}$ such that 
\begin{equation}\label{eq:basic-conv}
	\begin{split}
		& \vrho_{h} \overset{h\to 0}{\weak} \vrho, \quad 
		\text{in $L^\infty(0,T;L^\gamma(\Omega))
		\cap L^{2\gamma}((0,T)\times \Omega)$},\\
		& \vc{w}_{h} \overset{h\to 0}{\weak} \vc{w}, \quad 
		\text{in $L^\infty(0,T;\vc{L}^2(\Om))\cap L^2(0,T;\Hcurl)$}, \\
		& \vc{u}_{h} \overset{h\to 0}{\weak} \vc{u}, \quad 
		\text{in $L^\infty(0,T;\vc{L}^2(\Om))\cap L^2(0,T;\Hdiv)$}.
	\end{split}
\end{equation}
Furthermore, using the standard Hodge decomposition $\vu = \Curl \vc{\zeta} + \Grad s$ and 
orthogonality, 
\begin{equation}\label{eq:basic-conv2}
	\begin{split}
		\vc{z}_h \overset{h\to 0}{\weak} \Grad s, \quad 
		\text{in $L^\infty(0,T;\vc{L}^2(\Om))\cap L^2(0,T;\Hdiv)$}, \\
		\Curl \vc{\zeta}_h \overset{h\to 0}{\weak} \Curl \vc{\zeta}, \quad 
		\text{in $C(0,T;\vc{L}^2(\Om))\cap W^{1,2}(0,T;\vc{L}^2(\Om))$}.
	\end{split}
\end{equation}
In addition,
\begin{equation*}
	\vrho_h^\gamma \overset{h\to 0}{\weak}\overline{\vrho^\gamma}, 
	\quad 
	\vrho_h^{\gamma+1} 
	\overset{h\to 0}{\weak}
	\overline{\vrho^{\gamma+1}}, 
	\quad
	\vrho_h\log\vrho_h \overset{h\to 0}{\weak} \overline{\vrho\log\vrho},
\end{equation*}
where each $\overset{h\to 0}\weak$ signifies weak convergence 
in a suitable $L^p$ space with $p>1$.

Finally, $\vrho_h$, $\vrho_h\log\vrho_h$ converge 
respectively to $\vrho$, $\overline{\vrho\log\vrho}$ 
in $C([0,T];L^p_{\text{weak}}(\Om))$ for some 
$1<p<\gamma$, cf.~Lemma \ref{lem:timecompactness} 
and also \cite{Feireisl:2004oe,Lions:1998ga}. 
In particular, $\vrho$, $\vrho\log \vrho$, and 
$\overline{\vrho\log\vrho}$ belong to $C([0,T];L^p_{\text{weak}}(\Om))$.

\subsection{Strong convergence of the velocity}

\begin{lemma} \label{lemma:velocitycompactness} \solutiontext
Then
$$
	\vc{u}_{h} \rightarrow \vc{u}, \quad \text{in } L^2(0,T; \vc{L}^2(\Omega)).
$$
\end{lemma}
\begin{proof}
By virtue of \eqref{eq:basic-decomp} we can consider
each component of the decomposition $\vu_h = \Curl \vc{\zeta}_h + \vc{z}_h$. 
In Lemma \ref{lemma:curlcompact} below we prove that
\begin{equation*}
	\Curl \vc{\zeta}_h \rightarrow \Curl \vc{\zeta}, \quad \text{in $L^2(0,T;\vc{L}^2(\Om))$},
\end{equation*}
and hence it only remains to prove that $\vc{z}_h \rightarrow \vc{z}$ in the sense of 
distributions. 

From Lemma \ref{lemma:velocitycont}, we have the the weak time-continuity estimate:
$$
	\Pth{\vc{z}_{h}} \inb L^{2}(0,T;W^{-1,1}(\Omega)).
$$
Lemma \ref{lemma:spacetranslation} provides the spatial translation estimate:
$$
	\|\vc{z}_{h}(t,x) - \vc{z}_{h}(t, x - \xi)\|_{L^2(0,T;\vc{L}^2(\Omega))} \leq C(|\xi|^2 + |\xi|^\frac{4-N}{2})^\frac{1}{2}\|\Div \vc{z}_{h}\|_{L^2(0,T;L^2(\Omega))},
$$
where the constant $C>0$ is independent of $h$ and $\xi$. 
Lemma \ref{lemma:aubinlions} can then be applied (recalling \eqref{eq:basic-conv2}) to obtain the desired result;
$$
	\vc{z}_{h} \rightarrow \Grad s, \quad \textrm{in }L^2(0,T;\vc{L}^2(\Omega)).
$$
\end{proof}

\begin{lemma}\label{lemma:curlcompact}
	Given \eqref{eq:basic-conv} and \eqref{eq:basic-conv2},
	 % define $\Set{(\vc{\zeta}_{h},\vc{z}_{h})}_{h>0}$ in terms 
	 % 	of the decomposition $\vc{u}_{h}(t,\cdot) = 
	 % 	\Curl \vc{\zeta}_{h}(t,\cdot) + \vc{z}_{h}(t,\cdot)$ with 
	 % 	$\vc{\zeta}_{h}(t,\cdot) \in \vc{W}_{h}^{0,\perp}$, 
	 % 	$\vc{z}_{h}(t,\cdot) \in \vc{V}_{h}^{0,\perp}$, $t \in (0,T)$.
	\begin{equation*}%\label{eq:wh-curlwh-strongconv}
		\vc{w}_{h}  \overset{h\to 0}{\to} \vc{w}, \quad 
		\Curl\vc{\zeta}_{h} \overset{h\to 0}{\to} \Curl \vc{\zeta} 
		\quad \text{in $L^2(0,T;\vc{L}^2(\Omega))$.}
	\end{equation*}
\end{lemma}

\begin{proof}
	Fix any $t \in (0,T)$ and $m$ such that $t \in (t^{m-1}, t^m]$,
	where $t^m= m\Delta t $.
	Subtract the first equation of \eqref{FEM:momentumeq} 
	with $\vc{v}_h = \Curl \vc{\xi}_h^m$ from $\mu$ times the 
	second equation of \eqref{FEM:momentumeq}. 
	Multiplying the result with $\Delta t$ and summing 
	over all $k=1,\ldots,m$ yields 
	\begin{equation}\label{eq:curlconv1}
		\begin{split}
			&\int_0^t \int_\Om \mu \Curl \vc{\eta}_h \Curl \vc{\zeta}_h 
			- \mu \Curl \vc{w}_h \Curl \vc{\xi}_h \ dxdt \\
			&\qquad = \int_0^t\int_\Om \mu \vc{w}_h\vc{\eta}_h
			+ \Pth{\Curl \vc{\zeta}_h} \Curl \vc{\xi}_h \ dxdt,
		\end{split}
	\end{equation}
	for all $\vc{\eta}_h, \vc{\xi}_h$ that are piecewise constant 
	in time with values in $\vc{W}_h(\Om)$. 
	Fixing $\vc{\eta},\vc{\xi} \in C_c^\infty((0,T)\times \Om)$, we use 
	in \eqref{eq:curlconv1} the test functions
	\begin{align*}
		&\vc{\eta}_h(t,\cdot) =\vc{\eta}_h^m(\cdot) 
		:=\frac{1}{\Delta t}\int_{t^{m-1}}^{t^m} 
		\Pi_h^W \vc{\eta}(\cdot, s) \ ds,
		\quad \text{$t\in (t_{m-1},t_m)$, $m=1,\ldots,M$.}
		\\ &  \vc{\xi}_h(t,\cdot) = \vc{\xi}_h^m(\cdot) 
		:=\frac{1}{\Delta t}\int_{t^{m-1}}^{t^m} 
		\Pi_h^W \vc{\xi}(\cdot, s) \ ds,
		\quad \text{$t\in (t_{m-1},t_m)$, $m=1,\ldots,M$.} 
	\end{align*}
	Due to Lemma \ref{lemma:interpolation}, 
	$\Curl \vc{\xi}_h \to \Curl \vc{\xi}$ and 
	$\Curl \vc{\eta}_h \to \Curl \vc{\eta}$ in $L^2(0,T;\vc{L}^2(\Omega))$. 
	As a consequence, keeping in mind \eqref{eq:basic-conv} and \eqref{eq:basic-conv2}, 
	we let $h\to 0$ in \eqref{eq:curlconv1} to obtain
	\begin{equation}\label{eq:curlconv2}
		\begin{split}
			&\int_0^t \int_\Om \mu \Curl \vc{\eta} \Curl \vc{\zeta}
			-\mu \Curl \vc{w} \Curl \vc{\xi} \ dxdt \\
			&\qquad = \int_0^t\int_\Om \mu \vc{w}\vc{\eta}
			+\partial_t(\Curl \vc{\zeta}) \Curl \vc{\xi} \ dxdt,
			\quad \forall \vc{\eta},\vc{\xi} \in \vc{C}_c^\infty((0,T)\times \Om).
		\end{split}
	\end{equation}

	%Density is Theorem 2.6 in the first (or second) edition.
	Since $\vc{C}_c^\infty((0,T)\times \Om)$
	is dense in $L^2(0,T;\vc{W}^{\Curl, 2}_0(\Om))$ (\cite{Girault:1986fu}), we 
	see that \eqref{eq:curlconv2} holds for all 
	$\vc{\eta},\vc{\xi} \in L^2(0,T;\vc{W}^{\Curl, 2}_0(\Om))$.
	Hence, taking $\vc{\eta} = \vc{w}$, $\vc{\xi}=\vc{\zeta}$ 
	in \eqref{eq:curlconv2},
	\begin{equation}\label{eq:whatwewant}
		\frac{1}{2}\int_\Om |\Curl \vc{\zeta}^0|^2~dx = \int_0^t \int_\Om \mu \abs{\vc{w}}^2~dxdt
		+ \frac{1}{2}\left(\int_\Om |\Curl \vc{\zeta}|^2~dx\right)(t),
	\end{equation}
where $\Curl \vc{\zeta}^0$ is given by the Hodge decomposition $\vu^0 = \Curl \vc{\zeta}^0 + \Grad s^0$.

	Next, setting $\vc{\eta}_h = \vc{w}_h$ and 
	$\vc{\xi}_h = \vc{\zeta}_h$ 
	in \eqref{eq:curlconv1}, we observe that
	\begin{equation}\label{eq:disc-whatwewant}
		\begin{split}
			\frac{1}{2}\int_\Om |\Curl \vc{\zeta}_h^0|^2~dx 
			&= \int_0^t\int_\Om \mu \abs{\vc{w}_h}^2~dxdt \\
			&\qquad +\frac{1}{2}\left(\int_\Om |\Curl \vc{\zeta}_h|^2~dx\right)(t) + \frac{1}{2}\sum_{m=1}^M\int_\Om \jump{\Curl \vc{\zeta}_h^m}^2~dx.
		\end{split}
	\end{equation}
	
	Subtracting \eqref{eq:whatwewant} from \eqref{eq:disc-whatwewant} 
	\begin{equation*}
		\begin{split}
			&\lim_{h \rightarrow 0}\left[
			\int_0^t\int_\Om \mu (\abs{\vc{w}_h}^2- |\vc{w}|^2)~dxdt 
			+ \frac{1}{2}\left(\int_\Om|\Curl \vc{\zeta}_h|^2 - |\Curl \vc{\zeta}|^2 ~dx \right)(t)
			\right]\\
			&\qquad \leq \lim_{h \rightarrow 0}\left[
				\frac{1}{2}\int_\Om |\Curl \vc{\zeta}_h^0|^2 - |\Curl \vc{\zeta}^0|^2~dx 
				\right] = 0,
		\end{split}
	\end{equation*}
	where the last equality is an application of Lemma \ref{lemma:hodge}.
	Consequently, for any $t \in (0,T)$,
	\begin{equation*}
		\begin{split}
			&\lim_{h \rightarrow 0}\left[
			\mu\|\vc{w}_h - \vc{w}\|_{L^2(0,t;\vc{L}^2(\Om))}^2 + \frac{1}{2}\|\Curl \vc{\zeta}_h(t, \cdot) - \Curl \vc{\zeta}(t, \cdot)\|_{\vc{L}^2(\Om)}^2
			\right] \\
			&\qquad = 
			\lim_{h \rightarrow 0} \left[\int_0^t\int_\Om \mu\left( |\vc{w}|^2 - \vc{w}_h \vc{w}\right)~dxdt\right]  \\
			&\qquad \qquad + \lim_{h\rightarrow 0}\left(\int_\Om |\Curl \vc{\zeta}|^2 - (\Curl \vc{\zeta}_h)(\Curl \vc{\zeta})~dx\right)(t)
			 = 0,
		\end{split}
	\end{equation*}
	where the last term converges to zero due to the weak convergences \eqref{eq:basic-conv2}.	
% 	Letting $h\to 0$, keeping in mind the strong convergence $\Curl \vc{\zeta}^0_h$, and comparing the result 
% 	with \eqref{eq:whatwewant} reveals that
% \begin{equation*}
% 	\begin{split}
% 		&\lim_{h \to 0}\left[\frac{1}{2}\left(\int_\Om |\Curl \vc{\zeta}_h|^2~dx\right)(t) +\int_0^t\int_\Om 
% 		\mu\abs{\vc{w}_h}^2 \ dxdt \right] \\
% 		&\qquad \qquad \leq
% 		\frac{1}{2}\left(\int_\Om |\Curl \vc{\zeta}|^2~dx\right)(t)+ \int_0^t\int_\Om \mu\abs{\vc{w}}^2 \ dxdt.
% 	\end{split}
% \end{equation*}
%  	By virtue of Lemma \ref{lem:prelim}, the above inequality must in fact be equality and
% 	the proof follows.
\end{proof}

In the subsequent analysis, we will need the following technical lemma. 
For notational convenience, we define the linear time interpolant $\Pi_{\mathcal{L}}$:
\begin{equation}\label{eq:timeint}
\left(\Pi_{\mathcal{L}}f\right)(t) = f^{m-1} + \frac{t - t^{m-1}}{\Delta t}(f^{m} - f^{m-1}), \quad t \in (t^{m-1}, t^m).
\end{equation}

\begin{lemma}\label{lemma:samelimit}
Given \eqref{eq:basic-conv} and \eqref{eq:basic-conv2},
\begin{align*}
\Delta^{-1}[\Div \vc{u}_{h}] &\overset{h \rightarrow 0}{\rightarrow } \Delta^{-1}[\Div \vc{u}], \quad \textrm{in } L^2(0,T;\vc{W}^{1,2}(\Omega)), \\
\Delta^{-1}[\Div \Pi_{\mathcal{L}}\vc{u}_{h}] &\overset{h \rightarrow 0}{\rightarrow } \Delta^{-1}[\Div \vc{u}], \quad \textrm{in } L^2(0,T;\vc{W}^{1,2}(\Omega)).
\end{align*}
\end{lemma}
\begin{proof}

Recall the continuous Hodge decomposition
\begin{equation*}
	\vu = \Curl \vc{\zeta} + \Grad s.
\end{equation*}
As in the proof of Lemma \ref{lemma:embedding}, we have that $\vc{z}_h = \Aoph{\Div \vu_h}$. 
Hence, 
\begin{equation*}
	\begin{split}
			&\|\Aoph{\Div \vu_h} - \Aop{\Div \vu}\|_{L^2(0,T;\vc{L}^2(\Om))} \\
& \qquad = \|\vc{z}_h - \Grad s\|_{L^2(0,T;\vc{L}^2(\Om))}.
	\end{split}
\end{equation*}
From Lemma \ref{lemma:velocitycompactness}, we have that the right-hand side converges to 
zero. Hence, we conclude that $\Aoph{\Div \vu_h} \rightarrow \Aop{\Div \vu}$ in $L^2(0,T;\vc{L}^2(\Om))$.
Next, we write
\begin{equation*}
	\begin{split}
			&\|\Aop{\Div \vu_h}- \Aop{\Div \vu}\|_{L^2(0,T;\vc{L}^2(\Om))} \\
			&\qquad \leq \|\Aop{\Div \vu_h}- \Aoph{\Div \vu_h}\|_{L^2(0,T;\vc{L}^2(\Om))} \\
			&\qquad \qquad + \|\Aoph{\Div \vu_h}- \Aop{\Div \vu}\|_{L^2(0,T;\vc{L}^2(\Om))} \\
			&\qquad \leq Ch\|\Div \vu_h\|_{L^2(0,T;L^2(\Om))} \\
			&\qquad \qquad + \|\Aoph{\Div \vu_h}- \Aop{\Div \vu}\|_{L^2(0,T;\vc{L}^2(\Om))}.
	\end{split}
\end{equation*}
By sending $h \rightarrow 0$, we discover 
\begin{equation*}
	\Aop{\Div \vu_h} \rightarrow \Aop{\Div \vu} \quad \text{in }L^2(0,T;\vc{L}^2(\Om)),
\end{equation*}
which proves the first part of the lemma.
	
A direct calculation gives
\begin{equation*}
	\abs{\Pi_{\mathcal{L}}\vu_h(t, \cdot) - \vu_h(t, \cdot)}^2  \leq \abs{\jump{\vu_h^{k-1}(\cdot)}^2}, \quad t \in (t^{k-1}, t^k).
\end{equation*}
Hence, integrating over $(0,T) \times \Om$ yields
\begin{equation}\label{eq:jara}
	\norm{\Pi_{\mathcal{L}}\vu_h - \vu_h}_{L^2(0,T;\vc{L}^2(\Om))}^2 \leq \Delta t \sum_{k=1}^{M-1} \norm{\jump{\vu_h^k}}_{\vc{L}^2(\Om)}^2
	\leq C\Delta t,
\end{equation}	
where the last inequality follows from Lemma \ref{lemma:stability}.

Using elliptic theory and \eqref{eq:jara}, we conclude
\begin{equation*}
	\norm{\Delta^{-1}\left[\Pi_{\mathcal{L}}\vu_h - \vu_h \right]}_{L^2(0,T;\vc{W}^{1,2}(\Om))} \leq C(\Delta t)^\frac{1}{2}.
\end{equation*}	
Hence, the limits are equal and consequently the second part of the lemma 
now follows from the first.

\end{proof}

\subsection{Density scheme}

Having established strong convergence of the velocity we now
prove that the numerical solutions converge to 
a weak solution of the continuity equation \eqref{eq:contequation}.

\begin{lemma}[Convergence of the continuity approximation] \label{lemma:densityconv}
	The limit pair $(\vrho,\vc{u})$ constructed in \eqref{eq:basic-conv} 
	is a weak solution of the continuity equation \eqref{eq:contequation} in 
	the sense of  Definition \ref{def:weak}.
\end{lemma}

\begin{proof}
The proof is essentially identical to the proof of Lemma 6.4 in \cite{Karlsen1}
and is only included for the sake of completeness. 

Fix a test function $\phi \in C_{c}^\infty([0,T)\times\cOm)$, and 
introduce the piecewise constant approximations 
$\phi_{h}:=\Pi_{h}^Q \phi$, $\phi_{h}^m:= \Pi_{h}^Q \phi^m$, and
$\phi^m:=\frac{1}{\Delta t}\int_{t^{m-1}}^{t^m} \phi(t,\cdot)\ dt$. 

Let us employ $\phi^m_{h}$ as test function in 
the continuity scheme \eqref{FEM:contequation} 
and sum over $m=1,\ldots,M$.  The resulting equation reads
\begin{align*}
	&\sum_{m=1}^M \Delta t\int_{\Omega}\Pth{\vr_h^m}\phi_{h}^m\ dxdt \\
	& \quad = \sum_{\Gamma \in \Gamma^I_{h}}\sum_{m=1}^M \Delta t
	\int_\Gamma \left(\vrho^m_-(\vc{u}^{m}_h \cdot \nu)^+ 
	+ \vrho^m_+(\vc{u}^{m}_h \cdot \nu)^-\right)
	\jump{\phi^m_{h}}_\Gamma\ dS(x).
\end{align*}
As in the proof of Lemma \ref{lemma:timecont} we can rewrite this as
\begin{equation}\label{eq:contconvs}
	\begin{split}
		&\sum_{m=1}^M \Delta t\int_{\Omega} \Pth{\vr_h^m}\phi_{h}^m\ dxdt \\
		& = \sum_{m=1}^M \Delta t\int_{\Omega}
		\vrho^m_{h}\vc{u}^m_{h} \Grad \phi^m\ dx
		\\ & \qquad\quad
		+\sum_{E \in E_{h}}\sum_{m=1}^M \Delta t\int_{\binner}
		\jump{\vrho^m_{h}}_{\partial E}(\vc{u}^m_{h} \cdot \nu)^-
		(\phi^m_{h} - \phi^m)\ dS(x)\\
		& = \int_{0}^T \int_{\Omega}\vrho_{h}\vc{u}_{h}\Grad \phi\ dxdt 
		\\ & \qquad\quad
		+ \sum_{E \in E_{h}}\int_{0}^T \int_{\binner}
		\jump{\vrho_{h}}_{\partial E}(\vc{u}_{h}\cdot \nu)^-
		(\phi_{h} - \phi)\ dS(x)dt.
	\end{split}
\end{equation}

Lemma \ref{lemma:productionbound} tells us that
\begin{equation*}%\label{eq:contconv0}
	\abs{\sum_{E \in E_{h}}\int_{0}^T \int_{\binner}
	\jump{\vrho_{h}}_{\partial E}(\vc{u}_{h}\cdot \nu)^-
	(\phi_{h}-\phi)\ dS(x)dt} 
	\leq C\, h^\frac{1}{4}
	\norm{\Grad \phi}_{L^2(0,T;\vc{L}^{2^*}(\Omega))}.
\end{equation*}

In view of Lemma \ref{lemma:velocitycompactness},
$$
\lim_{h \rightarrow 0 }\int_{0}^T \int_{\Omega}\vrho_{h}\vc{u}_{h} \Grad \phi\ dxdt 
= \int_{0}^T \int_{\Omega}\vrho \vc{u} \Grad \phi\ dxdt.
$$

Summation by parts gives
\begin{equation*}%\label{eq:contconv1}
	\begin{split}
		& \sum_{m=1}^M \Delta t\int_{\Omega} \Pth{\vr_h^m}\phi_{h}^m\ dxdt
		\\ & \quad 
		= -\int_{\Delta t}^T \int_{\Omega} \vrho_{h}(t-\Delta t,x)
		\frac{\partial}{\partial t} \left(\Pi_{\mathcal{L}} \phi_{h}\right)\ dxdt 
		- \int_{\Omega}\vrho_{h}^0\phi_{h}^1\ dx
		\\ & \quad 
		\overset{h\to 0}{\to}
		-\int_{0}^T\int_{\Omega}\vrho \phi_{t}\ dxdt 
		- \int_{\Omega}\vrho_{0}\phi(0,x)\ dx.
	\end{split}
\end{equation*}
where \eqref{eq:basic-conv}, together with the strong convergence 
 $\vrho^0_h \overset{h \to 0}{\to} \vrho_0$, was used to pass to the limit. 
Summarizing, letting $h \to0$ in \eqref{eq:contconvs} delivers 
the desired result \eqref{eq:weak-rho}.
\end{proof}

\subsection{Strong convergence of the density approximation}
To obtain strong 
convergence of the density approximations $\vrho_h$, the main ingredient is 
a weak continuity property of the quantity $\eff(\vrho_{h},\vc{u}_{h})$.
To derive this property 
we use \eqref{eq:effvisc} and a
corresponding equation for the weak limit $\overline{\eff(\vr, \vu)}$.
We start by deriving the latter.

Let $\psi \in C^\infty_c(0,T)$ be arbitrary, 
set $\phi = \psi(\varrho - \avg{\vrho^0})$ in \eqref{eq:effvisc}, take the limit $h \rightarrow 0$,
and apply Lemmas \ref{lemma:interpolationerrorbound} and \ref{lemma:higherorderpressure} to obtain
\begin{equation}\label{eq:thebeast}
\begin{split}
&\lim_{h \rightarrow 0}\int_{0}^T \int_{\Omega}\eff(\varrho_{h}, \vc{u}_{h})\psi  (\varrho - \avg{\vrho^0}) \ dxdt \\
&\qquad \qquad \qquad 
= \lim_{h \rightarrow 0}\int_{0}^T \int_{\Omega} \Pth{\vc{u}_{h}}\psi  \Aop{\varrho- \avg{\vrho^0}}\ dxdt.
\end{split}
\end{equation}	
Since the operator $\Delta^{-1}$ is self-adjoint, we can integrate 
by parts to obtain
\begin{equation*}%\label{eq:st1}
	\begin{split}
		&\int_{0}^T \int_{\Omega} \Pth{\vc{u}_{h}}\psi  \Aop{\varrho- \avg{\vrho^0}}~ dxdt \\
		&= -\int_0^T\int_\Om \Pth{\Delta^{-1}[\Div \vc{u}_h]}\psi (\vr - \avg{\vr^0})~dxdt \\
		&= -\int_0^T\int_\Om \frac{\partial }{\partial t}\left(\Delta^{-1}[\Div \Pi_\mathcal{L}\vc{u}_h]\right)\psi (\vr - \avg{\vr^0})~dxdt,
	\end{split}
\end{equation*}
where the last equality follows by definition of $\Pi_\mathcal{L}$ \eqref{eq:timeint}.

Next, we move $\psi$ inside the time integration and use that $\avg{\vr^0}$ is independent
 of time. This gives
\begin{equation}\label{forkenneth2}
	\begin{split}
		&\int_{0}^T \int_{\Omega} \Pth{\vc{u}_{h}}\psi  \Aop{\varrho- \avg{\vrho^0}}~ dxdt \\
		&= -\int_0^T\int_\Om \frac{\partial }{\partial t}\left(\psi\Delta^{-1}[\Div \Pi_\mathcal{L}\vc{u}_h]\right)\vr~dxdt \\
		&\qquad + \int_0^T\int_\Om \psi'(t)\left(\Delta^{-1}[\Div \Pi_\mathcal{L}\vc{u}_h]\right)(\vr - \avg{\vr^0})~dxdt.
	\end{split}
\end{equation}

At this point, we recall that $(\vr, \vu)$ is a weak solution to the continuity equation (Lemma \ref{lemma:densityconv}).
Inserting $\phi = \psi \Delta^{-1}[\Div \Pi_\mathcal{L}\vu_h]$ as test function in the weak form of continuity 
equation gives
\begin{equation*}
	\int_0^T\int_\Om \frac{\partial }{\partial t}\left(\psi\Delta^{-1}[\Div \Pi_\mathcal{L}\vc{u}_h]\right)\vr~dxdt
	= - \int_0^T\int_\Om \psi \vr \vu \Aop{\Pi_\mathcal{L}\vu_h}~dxdt
\end{equation*}
Setting this into \eqref{forkenneth2} gives
\begin{equation*}
	\begin{split}
		&\int_{0}^T \int_{\Omega} \Pth{\vc{u}_{h}}\psi  \Aop{\varrho- \avg{\vrho^0}}~ dxdt \\
		&\quad =  \int_{0}^T \int_{\Omega}  \psi \varrho \vc{u} \left(\Aop{\Div \Pi_{\mathcal{L}}\vc{u}_{h}}\right)
		+ \psi'(t) \Delta^{-1}[\Div \Pi_{\mathcal{L}}\vu_h] (\varrho-\avg{\vrho^0})~dxdt.
	\end{split}
\end{equation*}
Sending $h \rightarrow 0$ and applying Lemma \ref{lemma:samelimit} 
\begin{equation*}
\begin{split}
&\lim_{h \rightarrow 0}\int_{0}^t \int_{\Omega} \Pth{\vc{u}_{h}} \psi \Aop{\varrho-\avg{\vrho^0}}\ dxdt \\
&\quad
 =  \int_{0}^t \int_{\Omega}  \psi\varrho \vc{u} \left(\Aop{\Div \vc{u}}\right)
+ \psi'(t) \Delta^{-1}[\Div \vu] (\varrho-\avg{\vrho^0})~dxdt.
\end{split}
\end{equation*}
Finally, we insert this expression in \eqref{eq:thebeast} and obtain
\begin{equation}\label{eq:conteff}
\begin{split}
&\int_{0}^t\int_{\Omega}\overline{\eff(\varrho, \vc{u})} \varrho\psi ~ dxdt \\
&\qquad \quad =  \int_{0}^t \int_{\Omega} \psi\varrho \vc{u} \left(\Aop{\Div \vc{u}}\right)+ \psi'(t) \Delta^{-1}[\Div \vu] (\varrho-\avg{\vrho^0})~dxdt.
\end{split}
\end{equation}

\begin{lemma}[Effective viscous flux] \label{lemma:effectiveflux} 
Given the convergences in \eqref{eq:basic-conv},
\begin{equation*}
\begin{split}
\lim_{h \rightarrow 0 }\int_{0}^T\int_{\Omega}\psi\eff(\varrho_{h}, \vc{u}_{h})\varrho_{h}\ dxdt 
= \int_{0}^T\int_{\Omega}\psi\overline{\eff(\varrho, \vc{u})}\varrho\ dxdt,
\end{split}
\end{equation*}
for all $\psi \in C^1_{c}(0,T)$.
\end{lemma}

\begin{proof}
Let $\psi \in C^1_{c}(0,T)$ be arbitrary and set $\phi = \psi(\varrho_{h} - \avg{\vrho^0})$ in \eqref{eq:effvisc} to obtain
\begin{equation}\label{eq:convvisc1}
\begin{split}
&\lim_{h \rightarrow 0 }\int_{0}^T\int_{\Omega}\psi P_{f}(\varrho_{h}, \vc{u}_{h})\varrho_{h}\ dxdt
= \lim_{h\rightarrow 0}\int_{0}^T \int_{\Om} \Pth{\vc{u}_{h}}\psi \Aop{\varrho_{h} - \avg{\vrho^0}}dxdt,
\end{split}
\end{equation} 
where we have also used Lemmas \ref{lemma:higherorderpressure} and \ref{lemma:interpolationerrorbound}.
As in \eqref{eq:pressure3} and \eqref{eq:pressure4} we can use summation by parts and the continuity scheme \eqref{FEM:contequation} to deduce
the following equality for the   the right--hand side:
\begin{equation*}
	\begin{split}
		&\int_{0}^T \int_{\Om} \Pth{\vc{u}_{h}}\psi \Aop{\varrho_{h} - \avg{\varrho_{h}^0}}dxdt \\
	 	&= -\sum_{m=1}^{M} \Delta t \int_{\Omega}\vc{u}_{h}^{m-1}\Pth{\psi^{m}} \Aop{\varrho_{h}^{m} -\avg{\varrho_{h}^0}} 
			+ \vc{u}_{h}^{m-1}\psi^{m-1}\Aop{\Pth{\varrho_{h}^{m}}}\ dx \\
		&\qquad \qquad \qquad  -  \frac{1}{\Delta t}\int_{0}^{\Delta t} 
			\int_{\Omega}\psi \vc{u}_{h}^0\Aop{(\varrho_{h} - \avg{\varrho_{h}^0})} dxdt \\
		&= -\sum_{m=1}^{M} \Delta t \int_{\Omega}\vc{u}_{h}^{m-1}\Pth{\psi^{m}} \Aop{\varrho_{h}^{m} -\avg{\varrho_{h}^0}}\ dx \\
		&\qquad  \quad + \sum_{m=1}^M \Delta t \int_{\Om}\psi^{m-1} \vc{u}_{h}^m \varrho_{h}^m \Aop{\Div \vc{u}^{m-1}_{h}}\ dx \\
		&\qquad  \qquad \quad + \sum_{m=1}^M \Delta t \sum_{E \in E_{h}}\int_{\partial E}
					\psi^m\jump{\varrho_{h}^m}_{\partial E}(\vc{u}^m_{h} \cdot \nu)^-(\Pi_{h}^V - \mathbb{I})\Delta^{-1}(\Div \vc{u}^{m-1}_{h})\ dS(x) \\
		& \qquad\quad		-  \frac{1}{\Delta t}\int_{0}^{\Delta t} \int_{\Omega}\psi \vc{u}_{h}^0\Aop{(\varrho_{h} - \avg{\varrho_{h}^0})}\ dxdt.
	\end{split}
\end{equation*}
Taking the limit $h \rightarrow 0$ and applying Lemma \ref{lemma:productionbound} gives
\begin{equation}\label{eq:notsoeasy}
	\begin{split}
	&\lim_{h \rightarrow 0}\int_{0}^T \int_{\Om}\Pth{\vc{u}_{h}}\psi \Aop{\varrho_{h} - \avg{\varrho_{h}^0}}dxdt \\
	&=\lim_{h \rightarrow 0}\sum_{m=1}^M \Delta t \int_{\Om}\psi^{m-1} \vc{u}_{h}^m \varrho_{h}^m \Aop{\Div \vc{u}^{m-1}_{h}}\ dx  \\
	&\qquad \qquad  - \int_{0}^T\int_{\Omega}\vc{u}\psi'(t)\Aop{\varrho - \avg{\vrho^0}}\ dxdt. 
	\end{split}
\end{equation}
We will now pass to the limit in the first term on the right-hand side. 

From Lemmas \ref{lemma:higherorderpressure}, \ref{lemma:velocitycompactness}, and \ref{lemma:samelimit}, we
have that
\begin{align}
	\vr_h &\weak \vr\quad \text{in }L^\infty(0,T;L^\gamma(\Om))\cap L^{\gamma+1}(0,T;L^{\gamma+1}(\Om)), \nonumber\\
	\vu_h &\rightarrow \vu \quad \text{in }  L^2(0,T;\vc{L}^2(\Om)), \label{eq:raj}\\
	\Aop{\Div \vu_h} &\rightarrow \Aop{\Div\vu} \quad \text{in }  L^2(0,T;\vc{L}^2(\Om)).\label{eq:rajraj}
\end{align}
This is insufficient to pass to the limit in the desired term. 
However, since $\vu_h \in_b L^\infty(0,T;\vc{L}^2(\Om))\cap  L^2(0,T;\vc{L}^{2^*}(\Om))$, we can 
by similar arguments as in the proof of Lemma \ref{lemma:samelimit} deduce that
$\Aop{\Div \vu_h} \in_b L^\infty(0,T;\vc{L}^2(\Om))\cap  L^2(0,T;\vc{L}^{2^*}(\Om))$. 
Let $\beta$ be given by
$$
\frac{2}{\beta} = \frac{1}{2} + \frac{1}{2^*}.
$$
Then, $\beta \geq \frac{2N}{N-1} - \epsilon$, for any $\epsilon > 0$. Since $2 \leq \beta$, 
the standard interpolation inequality can be applied and yields
\begin{equation*}
	\begin{split}
			\int_0^T \|f\|_{L^\beta(\Om)}^4~dt &\leq \int_0^T \|f\|^2_{L^2(\Om)}\|f\|^2_{L^{2^*}(\Om)}~dt \\
			 &\leq \|f\|^2_{L^\infty(0,T;L^2(\Om))}\|f\|^2_{L^2(0,T;L^{2^*}(\Om))}.
	\end{split}
\end{equation*}
From this inequality, we conclude
\begin{equation}\label{eq:inta}
		\Aop{\Div \vu_h},\vu_h \inb L^4(0,T;L^{\beta}(\Om)),
\end{equation}
For notational convenience, we introduce the function $g_h$ 
\begin{equation*}
	g_h(t, \cdot) = \vu_h(t, \cdot)\cdot\Grad \Delta^{-1}\left[\Div \vu_h(t-\Delta t, \cdot)\right].
\end{equation*}
Note that $g_h$ is precisely the scalar product in \eqref{eq:notsoeasy}. 
From the H\"older inequality and \eqref{eq:inta}, we have in particular that
\begin{equation*}
	g_h \inb L^2(0,T;L^{2}(\Om))
\end{equation*}
This, together with \eqref{eq:raj} and \eqref{eq:rajraj}, tells us that
\begin{equation*}
	g_h \rightarrow g := \vu \Aop{\Div \vu}, \quad \text{in $L^p(0,T;L^p(\Om))$, for any $p<2$, as $h \rightarrow 0$}.
\end{equation*}
Hence, $g_h\vr_h \weak g\vr$ in the sense of distributions on $(0,T) \times \Om$.
This is sufficient to pass to the limit in the first term on the right-hand side 
of \eqref{eq:notsoeasy}. By sending $h \rightarrow 0$ in \eqref{eq:notsoeasy}, we obtain the identity
\begin{equation}\label{eq:convvisc2}
\begin{split}
&\lim_{h \rightarrow 0}\int_{0}^T \int_{\Om}\Pth{\vc{u}_{h}}\psi \Aop{\varrho_{h} - \avg{\varrho_{h}^0}}dxdt \\ 
&\qquad \quad = \int_{0}^T\int_{\Omega}\psi \varrho \vc{u} \Aop{\Div \vc{u}} - \vc{u}\psi'(t)\Aop{\varrho - \avg{\vrho^0}}\ dxdt.
\end{split}
\end{equation}
Then, \eqref{eq:convvisc2} in \eqref{eq:convvisc1} yields
\begin{equation*}
\begin{split}
&\lim_{h \rightarrow 0 }\int_{0}^T\int_{\Omega}\eff(\varrho_{h}, \vc{u}_{h})\psi \varrho_{h}\ dxdt \\
&\qquad =\int_{0}^T\int_{\Omega}\psi \varrho \vc{u} \Aop{\Div \vc{u}} - \vc{u}\psi'(t)\Aop{\varrho - \avg{\vrho^0}}\ dxdt \\
&\qquad =  \int_{0}^T \int_{\Omega} \psi \varrho \vc{u} \left(\Aop{\Div \vc{u}}\right) 
		+ \psi'(t) (\varrho- \avg{\vrho^0}) \left(\Delta^{-1}\Div \vc{u}\right) \ dxdt \\
&=\int_{0}^T\int_{\Omega}\overline{\eff(\varrho, \vc{u})}\psi \varrho\ dxdt,
\end{split}
\end{equation*}
where the last equality is \eqref{eq:conteff}. This concludes the proof.
\end{proof}

We are now in a position to prove strong convergence of the density approximations. 
\begin{lemma}[Strong convergence of the density]\label{lem:strongconv}
	Suppose that \eqref{eq:basic-conv} holds. Then, passing to 
	a subsequence if necessary,
	$$
	\vrho_{h} \rightarrow \vrho
	\quad \text{a.e.~in~$(0,T)\times \Omega$.}
	$$
\end{lemma}

\begin{proof}	
The proof is identical to that of Lemma 6.6 in \cite{Karlsen1} and
is included for the sake of completeness. 	
	
In view of Lemma \ref{lemma:densityconv}, the 
limit $(\vrho,\vc{u})$ is a weak solution 
of the continuity equation and hence, by Lemma \ref{lemma:feireisl}, 
also a renormalized solution. In particular, 
$$
\left(\vrho\log \vrho\right)_t 
+ \Div \left( \left(\vrho\log\vrho\right)
\vc{u}\right)=\vrho \Div \vc{u} 
\quad \text{in the weak sense on $\cDom$.}
$$

Since $t\mapsto \vrho\log \vrho$ is continuous 
with values in some Lebesgue space 
equipped with the weak topology, we can use this 
equation to obtain for any $t>0$
\begin{equation}\label{eq:stngdenconv-eq1}
	\int_{\Omega} \left(\vrho \log \vrho\right)(t)\ dx
	-\int_{\Omega}\vrho_{0}\log \vrho_{0}\ dx
	= -\int_{0}^t \int_{\Omega}\vrho \Div \vc{u}\ dxds
\end{equation}

Next, we specify $\phi_h\equiv 1$ as test function in the 
renormalized scheme \eqref{FEM:renormalized}, multiply by $\Delta t$,
and sum the result over $m$. Making use of the 
convexity of $z\log z$, we infer for any $m=1,\dots,M$ 
\begin{equation}\label{eq:stngdenconv-eq2}
	\int_{\Omega}\vrho^m_{h}\log \vrho^m_{h}\ dx
	-\int_{\Omega}\vrho^{0}_h\log \vrho^{0}_h\ dx  
	\leq -\sum_{k=1}^m \Delta t\int_{\Omega}\vrho^m_{h}
	\Div \vc{u}^m_{h}\ dxdt.
\end{equation}

In view of the convergences stated at the beginning of this section 
and strong convergence of the initial data, we 
can send $h \to 0$ in \eqref{eq:stngdenconv-eq2} to obtain
\begin{equation}\label{eq:stngdenconv-eq3}
	\int_{\Omega} \Bigl(\overline{\vrho \log \vrho}\Bigr)(t)\ dx
	-\int_{\Omega}\vrho_{0}\log \vrho_{0}\ dx
	\le -\int_{0}^t \int_{\Omega}\overline{\vrho \Div \vc{u}}\ dxds.
\end{equation}

Subtracting \eqref{eq:stngdenconv-eq1} 
from \eqref{eq:stngdenconv-eq3} gives
\begin{align*}%\label{eq:stngdenconv-ineq}
	\int_{\Omega}\Bigl(\overline{\vrho \log \vrho}-\vrho \log \vrho\Bigr)(t)\ dx
	& \leq -\int_{0}^t\int_{\Omega}
	\overline{\vrho\Div \vc{u}}-\vrho \Div \vc{u}\ dxds,
\end{align*}
for any $t\in (0,T)$. Lemma \ref{lemma:effectiveflux} tells us that
\begin{equation*}%\label{eq:stngdenconv-ident}
	\int_{0}^t\int_{\Omega} \overline{\vrho \Div \vc{u}}
	-\vrho \Div \vc{u}\ dxds 
	= \frac{a}{\mu + \lambda}\int_{0}^t\int_{\Omega}
	\overline{\vrho^{\gamma +1}}
	-\overline{\vrho^\gamma} \vrho\ dxds\ge 0,
\end{equation*}
where the last inequality follows 
as in  \cite{Feireisl:2004oe,Lions:1998ga}, so 
the following relation holds: 
$$
\overline{\vrho \log \vrho}=\vrho \log \vrho
\quad \text{a.e.~in $\Dom$.}
$$
Now an application of Lemma \ref{lem:prelim} brings the proof to an end.
\end{proof}

\subsection{Velocity scheme}

\begin{lemma}[Convergence of the momentum approximation]\label{lemma:momentumconv}
	The limit triple $(\vc{w},\vc{u},\vrho)$ 
	constructed in \eqref{eq:basic-conv} is a weak solution of 
	the velocity equation \eqref{eq:momentumeq} in the sense of \eqref{def:mixed-weak}.
\end{lemma}

\begin{proof}
Fix $(\vc{v}, \vc{\eta}) \in \vc{C}^\infty_c((0,T)\times \Om)$, and introduce the projections 
$\vc{v}_{h} = \Pi_{h}^V \vc{v}$, $\vc{\eta}_{h} =\Pi_{h}^W \vc{\eta}$ and 
$\vc{v}_{h}^m = \frac{1}{\Delta t}\int_{t^{m-1}}^{t^m}\vc{v}_{h} \ dt$, 
$\vc{\eta}_{h}^m = \frac{1}{\Delta t}\int_{t^{m-1}}^{t^m}\vc{\eta}_{h} \ dt$. 

Utilizing $\vc{v}^m_{h}$ and $\vc{\eta}^m_{h}$ as test functions in 
the velocity scheme \eqref{FEM:momentumeq}, multiplying by $\Delta t$, 
 summing the result over $m$, and applying summation by parts,  we gather
\begin{equation}\label{eq:approx-mixed-weak}
	\begin{split}
		&-\int_{\Delta t}^T\int_{\Omega}\vu_h(t-\Delta t, x)\Pth{\vc{v}_h}~dxdt \\
		&\qquad + \mu\Curl \vc{w}_{h}\vc{v}_{h} 
		+ \left[(\mu + \lambda)\Div \vc{u}_{h}
		-p(\vrho_{h})\right]\Div \vc{v}_{h}\ dxdt
		= \int_\Om \vu_h^0 \vv_h^1~dx,\\
		&\int_0^T\int_{\Omega}\vc{w}_{h}\vc{\eta}_{h} 
		-\vc{u}_{h}\Curl \vc{\eta}_{h} \ dxdt=0.
	\end{split}
\end{equation}
In view of Lemma \ref{lemma:interpolation}, $\vc{v}_h \rightarrow \vc{v}$ in 
$L^\infty(0,T;\vc{W}^{\text{div},p})$ for any finite $p$ and 
$\vc{\eta}_h \rightarrow \vc{\eta}$ in $L^\infty(0,T;\vc{W}^{\text{curl},p})$.
Furthermore, by Lemmas \ref{lemma:higherorderpressure} and \ref{lem:strongconv}
$p(\vr_h) \rightarrow p(\vr)$ in $L^\alpha((0,T)\times \Om)$ for any 
$\alpha < \gamma +1$. Hence, we can send $h \rightarrow 0$ in 
\eqref{eq:approx-mixed-weak} to obtain that the limit constructed 
in \eqref{eq:basic-conv} satisfies \eqref{eq:weak-u} for all 
test functions $(\vc{v}, \vc{\eta}) \in \vc{C}^\infty_c((0,T)\times \Om)$.
Since $\vc{C}^\infty_c((0,T)\times \Om)$ is dense in both $L^2(0,T;\Hcurl)$
and $W^{1,2}(0,T;L^2(\Om))\cap L^2(0,T;\Hdiv)$ \cite{Girault:1986fu}
this concludes the proof.

\end{proof}

\end{document}